\documentclass[11pt,A4]{article}
\usepackage{fullpage}
\usepackage{amsmath}
\usepackage{amsfonts}
\usepackage{amssymb}
\usepackage{tikz-cd}
\usepackage{titlesec}
\usepackage{mathabx}
\usepackage{color}
\usepackage{mathrsfs}
\usepackage{hyperref}
\usepackage[all]{xy}
\usepackage{graphicx}
\usepackage{cite}
\usepackage{caption}
\usepackage{graphicx, subfig}
\usepackage{enumerate}
\usepackage[T1]{fontenc}
\usepackage[utf8]{inputenc}
\usepackage{authblk}
\usepackage{multirow}

\title{Mirror symmetry for special nilpotent orbit closures}

\author[1]{Baohua Fu}
\author[2]{Yongbin Ruan}
\author[3]{Yaoxiong Wen}
\affil[1]{AMSS, HLM and MCM, Chinese Academy of Sciences, 55 ZhongGuanCun East Road, Beijing,
100190, China and School of Mathematical Sciences, University of Chinese Academy of Sciences,
Beijing, China\\ bhfu@math.ac.cn}
\affil[2]{Institute for Advanced Study in Mathematics, Zhejiang University, Hangzhou 310058, China \\
ruanyb@zju.edu.cn}
\affil[3]{{Korea Institute for Advanced Study}, {{85 Hoegiro, Dongdaemun-gu, Seoul}, {02455}, {Republic of Korea}}\\ { y.x.wen.math@gmail.com}}

\date{}

\newcommand{\cM}{{\mathcal M}}
\newcommand{\cN}{{\mathcal N}}
\newcommand{\cO}{{\mathcal O}}
\newcommand{\cP}{{\mathcal P}}

\newcommand{\g}{{\mathfrak g}}
\newcommand{\p}{{\mathfrak p}}
\newcommand{\fl}{{\mathfrak l}}
\newcommand{\fu}{{\mathfrak u}}
\newcommand{\n}{{\mathfrak n}}

\newcommand{\bC}{\mathbb{C}}

\def\0{{\mathcal O}}

\newcommand{\D}{{\mathscr D}}
\newcommand{\C}{{\mathbb C}}

\newcommand{\Ind}{\operatorname{Ind}}

\renewcommand{\bf}[1]{\textbf{#1}}

\newtheorem{theorem}{Theorem}[section]
\newtheorem{lemma}[theorem]{Lemma}
\newtheorem{proposition}[theorem]{Proposition}
\newtheorem{corollary}[theorem]{Corollary}

\newtheorem{example}[theorem]{Example}
\newtheorem{conjecture}[theorem]{Conjecture}
\newtheorem{remark}[theorem]{Remark}
\newtheorem{expectation}[theorem]{Expectation}
\newtheorem{question}[theorem]{Question}
\newtheorem{problem}[theorem]{Problem}
\newenvironment{proof}{{\noindent\it Proof}\quad}{\hfill $\square$ \\ }

\begin{document}
\maketitle
\begin{abstract}
	Motivated by geometric Langlands, we initiate a program to study the mirror symmetry between nilpotent orbit closures of a semisimple Lie algebra and those of its Langlands dual. The most interesting case is $B_n$ via $C_n$. Classically, there is a famous Springer duality between special orbits. Therefore, it is natural to speculate that the mirror symmetry we seek may coincide with Springer duality in the context of special orbits. Unfortunately, such a naive statement fails. To remedy the situation, we propose a conjecture which asserts the mirror symmetry for certain parabolic/induced covers of special orbits. Then, we prove the conjecture for Richardson orbits and obtain certain partial results in general. In the process, we reveal some very interesting and yet subtle structures of these finite covers, which are related to Lusztig's canonical quotients of special nilpotent orbits. For example, there is a mysterious asymmetry in the footprint or range of degrees of these finite covers. Finally, we provide two examples to show that the mirror symmetry fails outside the footprint.	
	\end{abstract}

\tableofcontents

\section{Introduction}

\subsection{Motivation from geometric Langlands program}

Geometric Langlands program is an old subject of mathematics. There has been a re-examination of the subject from a physical perspective during the last twenty years by Gukov, Kapustin, and Witten \cite{KW07, GW08, GW}, which has attracted a lot of attention from the geometry community.  Geometrically, geometric Langlands is a mirror conjecture of two moduli spaces of stable Higgs bundles $\cM(C, G, Higgs)$ via $\cM(C, {}^LG, Higgs)$, where $C$ is a fixed Riemann surface. Reductive complex Lie groups $G,  {}^LG$ are related by Langlands duality. The above mirror symmetry conjecture has been under investigation for almost twenty years, and progress has been slow. During the last ten years, Gukov-Witten (\cite{GW08}) has proposed in physics a version of geometric Langlands for the so-called moduli space of parabolic Higgs bundles $\cM(C, G, Higgs, \alpha_1, \cdots, \alpha_k)$ where $\alpha_i$'s are adjoint orbits of the Lie algebra $\mathfrak{g} $. From the Gromov-Witten theory point of view, these orbits $\alpha_i$ can be viewed as an insertion at a marked point. The above story was further generalized to the case of {\em wild ramification}, whose mathematical investigation was barely started. Back to the parabolic or mild ramified case, it is clear that a pre-request to formulate a mirror symmetry conjecture of the moduli space of parabolic Higgs bundles is a mirror symmetry conjecture between these adjoint orbits/insertions! It can be viewed as an infinitesimal local version of geometric Langlands. This is exciting since adjoint orbits are central and yet {\em classical} objects in the sense that there is no need to construct any moduli space! This infinitesimal local version of geometric Langlands is the focus of this article.

From the mirror symmetry perspective, it is always exciting to look for more nontrivial mirror pairs. During the last thirty years, the majority of examples studied in the community are complete intersections in toric varieties. During the last decade, there has been a growing interest in going beyond these examples. For example, a great deal of effort was spent studying complete intersections in nonabelian GIT quotients. Examples that appeared in geometric Langlands, such as adjoint orbits, are of very different nature. Geometrically, they possess many amazing properties such as hyperKähler structure and are considered as some of the most interesting examples from the mirror symmetry point of view. Algebraically, they play a central role in representation theory. Therefore, they are natural bridges between representation theory and mirror symmetry!

There are many versions of mirror symmetry. In this article, we will restrict ourselves to topological mirror symmetry, i.e., Hodge numbers. As we need to take the closure of nilpotent orbits, which are singular, we will use Batyrev's notion of stringy Hodge numbers (\cite{Ba98}).  This is also the framework considered by \cite{HT}, where a topological mirror symmetry conjecture for moduli spaces of Higgs bundles of type ${\rm SL}_n$ and  ${\rm PGL}_n$ is proposed, which was recently proven in \cite{GWZ} via p-adic integration and in \cite{MS} via the decomposition theorem.

Nilpotent orbit closures (up to normalization) are important examples of symplectic singularities (\cite{Be00}), which have been extensively studied both from the algebro-geometric part and from the representation theoretic part. However, their stringy E-functions have never been studied before. Note that the stringy Hodge numbers are equal to the Hodge numbers of its crepant resolution (if exists), while any symplectic resolution is a crepant resolution. Hence, if symplectic resolutions exist, then we are comparing Hodge numbers. Symplectic resolutions for nilpotent orbits are worked out in \cite{Fu1}, which leads us to consider first the case of Richardson orbits, and then we generalize to special orbits.

We want to emphasize that our work is just the beginning of the program, and much more needs to be investigated in the future. For example, we do not yet know how to formulate a general conjecture beyond special orbits.

\subsection{Motivation from the geometry of special nilpotent orbits}

Let $\g$ be a simple complex Lie algebra and $G^{\rm ad}$ its adjoint group.  Let $\cN$ be the set of all nilpotent orbits in $\g$, which is parametrized by weighted Dynkin diagrams (and by partitions in classical types).

For a nilpotent orbit $\cO$ in $\g$ and an element $x \in \cO$, we denote by $G_x^{\rm ad}$ the stabilizer of $G^{\rm ad}$ at $x$ and by $(G_x^{\rm ad})^\circ$ its identity component. The quotient group $G_x^{\rm ad}/(G_x^{\rm ad})^\circ$, which  is independent of the choice of $x$, is called the component group of $\cO$, denoted by $A(\cO)$. It turns out that $A(\cO)$ is an elementary 2-group for classical types (and  trivial for $\g =\mathfrak{sl}_n$).

A fundamental result proven by Springer (\cite{Spr}) gives a correspondence between irreducible representations of the Weyl group of $G$ and admissible pairs $(\cO, \psi)$, where $\cO$ is a nilpotent orbit in $\g$ and $\psi$ is an irreducible representation of $A(\cO)$. Later on, Lusztig \cite{Lus79} introduced special representations of Weyl groups, and it turns out these special representations are assigned to nilpotent orbits with trivial $\psi$ via the Springer correspondence. This gives a remarkable set of special nilpotent orbits, denoted by $\cN^{\rm sp}.$

As written in \cite{Lu97}, special nilpotent orbits play a key role in several problems in representation
theory, but unfortunately, their definition is totally un-geometrical. For this reason, special nilpotent orbits are often regarded as rather mysterious objects. For this reason, we are motivated to unveil some of the purely geometrical properties of special unipotent classes from the viewpoint of mirror symmetry.

We denote by ${}^L \g$ the Langlands dual Lie algebra of $\g$. Let ${}^L\mathcal{N}^{sp}$ be the set of special nilpotent orbits in ${}^L\g$. Let $G$ and ${}^LG$ be Langlands dual groups with Lie algebras $\g$ and ${}^L\g$ respectively. As $G$ and ${}^LG$ have the same Weyl group, there exists a natural bijection between  $\mathcal{N}^{sp}$ and ${}^L\mathcal{N}^{sp}$ via Springer correspondence, which we call \emph{Springer dual map} and denoted by
\begin{align*}
	S: \mathcal{N}^{sp} \longrightarrow {}^L\mathcal{N}^{sp}.
\end{align*}
It is observed by Spaltenstein that this bijection is order-preserving and dimension-preserving (\cite{Spa82}). Similar to the notation for Langlands dual, we denote $S(\0)$ by ${}^S\0$.

As $\g = {}^L\g$ except for type $B_n$ or $C_n$, the Springer dual map $S$ is non-identical only when $G$ is of type $B_n$ or $C_n$. From now on, we will assume that $G={\rm Sp}_{2n}$, then  ${}^LG = {\rm SO}_{2n+1}$.  Note that in this case, $G$ is simply-connected while ${}^LG$ is adjoint.

For a special nilpotent orbit $\0$, we denote by $\overline{\cO}^{\mathcal{SP}}$ the special piece of $\0$, which consists of elements in $\overline{\0}$ not contained in any smaller special nilpotent orbit closure.  By \cite[Theorem 0.9]{Lu97}, the $G$-equivariant cohomology of $\overline{\cO}^{\mathcal{SP}}$ and the ${}^LG$-equivariant cohomology of the Springer dual special piece $\overline{{}^S\0}^{\mathcal{SP}}$ are isomorphic as vector spaces.  This unveils a nice geometric property shared by Springer dual nilpotent orbits. Are there more geometric relations between closures of a special nilpotent orbit $\cO$ and its Springer dual ${}^S \cO$?

\subsection{Guiding philosophy and Lusztig's canonical quotient}

Note that nilpotent orbit closures are generally non-normal; the normality problem is completely settled in \cite{KP82, So05} for classical types. We denote by $\widetilde{\0} \to \overline{\0}$ the normalization map, then $\widetilde{\0}$ has only symplectic singularities in the sense of Beauville \cite{Be00}, which is the singular counterpart of hyperK\"ahler manifolds. 

It follows that $\widetilde{\0}$ has only Gorenstein canonical singularities. Hence we may consider its stringy E-function as defined in \cite{Ba98}. A naive guess would be that $\widetilde{\cO}$ and $\widetilde{{}^S\cO}$ have the same stringy E-function. Actually, this fails due to the following enlightening example.
\begin{example} \label{e.r=1}
Consider the following special nilpotent orbits, which are Springer dual:
$$
\0_C:=\0_{[2^2,1^{2n-4}]} \subset \mathfrak{sp}_{2n}, \quad \quad \0_B := \0_{[3,1^{2n-2}]} \subset \mathfrak{so}_{2n+1}.
$$
By \cite{KP82}, both closures $\overline{\0}_C, \overline{\0}_B$ are normal. Note that $\overline{\0}_B$ admits a crepant resolution given by $T^*\mathbb{Q}^{2n-1}  \to \overline{\0}_B$, which gives the stringy E-function of $\overline{\0}_B$ (where $q=uv$):  
$$
{\rm E}_{\rm st} (\overline{\0}_B) = {\rm E}(T^* \mathbb{Q}^{2n-1}) = {\rm E}(\mathbb{Q}^{2n-1}) q^{2n-1}.
$$
On the other hand, we can construct a log resolution for $\overline{\0}_C$ and then compute its stringy E-function explicitly (cf. Proposition \ref{p.r=1C}).  It turns out that ${\rm E}_{\rm st} (\overline{\0}_C)$ is not a polynomial if $n \geq 3$, which  shows that $\overline{\0}_C$ and $\overline{\0}_B$ have different stringy E-functions for $n \geq 3$.
\end{example}

How to remedy this? A nice observation from \cite{BK94} is that there exists a double cover from the minimal nilpotent orbit closure $\overline{\0}_{min}  \subset \mathfrak{sl}_{2n}$ to $\overline{\0}_C$, while we have a crepant resolution $T^*\mathbb{P}^{2n-1} \to \overline{\0}_{min}$. Note that $\mathbb{P}^{2n-1} $ and $\mathbb{Q}^{2n-1}$ have the same E-polynomials.  This implies that $\overline{\0}_{min} $ (or double cover of $\overline{\0}_C$) and $\overline{\0}_B$ have the same stringy E-function! This is quite surprising as the issue of finite covers was completely ignored in the mirror symmetry literature from the physical side.

Motivated by this, we can now formulate our guiding philosophy by the following non-precise expectation:
\begin{expectation} \label{exp}
Consider a pair of Springer dual special nilpotent orbits $(\cO, {}^S\cO)$. Then there exist (possibly many)    $G$-equivariant (resp. ${}^LG$-equivariant) finite covers $X\to \overline{\cO} $ (resp. ${}^L X \to \overline{{}^S\cO}$) such that the pairs $(X, {}^L X)$ are Springer dual and mirrors one to the other in some sense.
\end{expectation}

The main difficulty is determining which equivariant finite covers we will consider.  As discussed in the first subsection, we will restrict ourselves to topological mirror symmetry in this paper, while it is definitely worthwhile to investigate more in the future.

 For a nilpotent orbit $\cO$, denote by $\mathcal{A}(\cO):= G_x/ G_x^{\circ} $ for $x \in \cO$. Notice that $G={\rm Sp}_{2n}$, so $\mathcal{A}(\cO) = \pi_1(\cO)$, and ${}^LG={\rm SO}_{2n+1}$, so $\mathcal{A}({}^S\cO)= A({}^S\cO)$ is the usual component group.
Note that $G$--equivariant finite covers of $\overline{\cO}$ and  those for $\overline{{}^S\cO}$ are controlled by $\mathcal{A}(\0)$ and $\mathcal{A}({}^S\0)$ respectively.

 For a special nilpotent orbit $\0$, Lusztig's canonical quotient $\bar{A}(\cO)$ is a quotient of $\mathcal{A}(\cO)$, which plays an important role in Lusztig's classification of unipotent representations of finite groups of Lie type. Surprisingly, Lusztig's canonical quotient $\bar{A}(\cO)$ also plays an essential role in our approach to Expectation \ref{exp}. In fact, the finite covers we will consider share a kind of seesaw property governed by Lusztig's canonical quotient. This, in turn, gives a geometrical interpretation of Lusztig's canonical quotient(Proposition \ref{p.LusztigCanonicalQuotient}), which is somehow surprising to us.

We should point out that in a previous draft \cite{S}, it is observed that Lusztig's canonical quotient also plays a crucial role in the duality of singularity types between adjacent special nilpotent orbits.

\subsection{Mirror symmetry for parabolic covers of Richardson orbits} \label{s.Ric}

A closed subgroup $P \subset G$ is called parabolic if the quotient variety $G/P$ is projective.  Note that $G/P$ only depends on the conjugacy class of $P$.
Every parabolic subgroup is conjugated to one containing the fixed Borel subgroup, which is called \emph{standard}. A standard parabolic subgroup can be described via a subset $\Theta$ of simple roots $\Delta$, denoted by $P_{\Theta}$, i.e., its Lie algebra $\mathfrak{p}_{\Theta}$ has a Levi decomposition $\mathfrak{p}_{\Theta}= \fl _{\Theta} \oplus \fu _{\Theta}$, with  $\fl _{\Theta}=\mathfrak{h} \oplus \sum_{\alpha \in \langle \Theta \rangle } \g_{\alpha} $ and $\fu_{\Theta} = \sum_{\alpha \in \Phi^{+} \backslash \langle \Theta \rangle^{+}} \g_{\alpha} $, where $\Phi^+$ is the set of positive roots in the root system $\Phi$, and $\langle \Theta \rangle$ denote the sub-root system generated by $\Theta$ and $\langle \Theta \rangle^+=\langle \Theta \rangle \cap \Phi^+$.

For a parabolic subgroup $P \subset G$, the natural $G$--action on $T^*(G/P)$ is Hamiltonian, which induces the moment map $T^*(G/P) \to \g^*\simeq \g$. It is proven by Richardson that the image is a nilpotent orbit closure, which gives a generically finite projective morphism $\nu_P: T^*(G/P) \to  \overline{\cO}$.  Such orbits are called Richardson orbits, and $P$ (or its conjugacy class) is called a polarization of $\cO$. For a Richardson orbit $\cO$, we denote by ${\rm Pol}(\cO)$ the set of conjugacy classes of its polarizations, which is a finite set.

By taking the Stein factorization of $\nu_P$, we get a $G$--equivariant finite cover $\mu_P: X_P \to \overline{\cO}$, which is called the {\em parabolic cover} of $\cO$ associated with the polarization $P$.  Note that all these constructions depend only on the conjugacy class of $P$. Hence we will only consider standard parabolic subgroups from now on. Remark also that $X_P$ is normal by definition, so if $\nu_P$ is birational, then $X_P$ is just the normalization of $\overline{\cO}$.

There exists a canonical bijection between simple roots of $G= {\rm Sp}_{2n}$ and those of its dual group ${}^LG= {\rm SO}_{2n+1}$, which induces a canonical bijection between standard parabolic subgroups, denoted by $P \mapsto {}^LP$. We call ${}^LP$  the Langlands dual parabolic subgroup of $P$. In a similar way, we define the Langlands dual of Lie algebras $\mathfrak{p}$, $\mathfrak{l}$ and $\mathfrak{u}$, denoted respectively by ${}^L \mathfrak{p}$, ${}^L \mathfrak{l}$ and ${}^L\mathfrak{u}$.

It is known that Richardson orbits are special.  Our first main theorem in this paper gives a satisfactory answer to Expectation \ref{exp} in the case of parabolic covers of Richardson orbits.
\begin{theorem} \label{t.Richardson}
Let $\cO$ be a Richardson orbit in $\g=\mathfrak{sp}_{2n}$ and $P$ a polarization of $\0$.  Let ${}^LP$ be the Langlands dual parabolic subgroup in ${}^LG$ and ${}^S \cO \subset {}^L \g = \mathfrak{so}_{2n+1}$ the Springer dual of $\cO$.  Then
\begin{itemize}
\item[(1)] The orbit ${}^S\cO$ is Richardson, and Lusztig's canonical quotients for $\cO$ and ${}^S\cO$ are isomorphic.
\item[(2)] The parabolic subgroup ${}^LP$ is a polarization of ${}^S\cO$. Let $\mu_{{}^LP}: X_{{}^LP} \to \overline{{}^S\cO}$ be the associated parabolic cover. Then
the two covers $\mu_P$ and $\mu_{{}^LP}$ enjoy the seesaw property governed by Lusztig's canonical quotient. Namely, we have
\begin{align*}
	\deg \mu_P \cdot \deg \mu_{{}^LP} = |\bar{A}(\cO)| = |\bar{A}({}^S\cO)|.
\end{align*}
\item[(3)] For any $P \in {\rm Pol}(\0)$, the pair $(X_P, X_{{}^LP})$ of parabolic covers is a mirror pair, in the sense that they have the same stringy E-function.
\end{itemize}
\end{theorem}

In fact, we have a more precise understanding of Lusztig's canonical quotient as follows: let $K(\0) \subset \mathcal{A}(\0)$ be the kernel of the quotient $\mathcal{A}(\0) \to \bar{A}(\0)$.  For any polarization $P \in {\rm Pol}(\0)$, the parabolic cover $X_P \to \overline{\0}$ corresponds to a finite subgroup $\mathcal{A}_P$ of $\mathcal{A}(\0)$.  Similar notations apply for the dual part, i.e., $\mathcal{A}({}^S\cO) := {}^LG_{x^{\prime}} / {}^LG_{x^{\prime}}^{\circ}$ for $x^{\prime} \in {}^S\cO$, etc. Then we have a geometrical realization of Lusztig's canonical quotient as follows:
\begin{proposition} \label{p.LusztigCanonicalQuotient}
(i) We  have $K(\0) \subset \mathcal{A}_P$ and for the Springer dual, we have $K({}^S\0) \subset \mathcal{A}_{{}^LP}$.

 (ii) Let $\bar{A}_P $ be the quotient $\mathcal{A}(\0)/\mathcal{A}_P$ and $\bar{A}_{{}^LP} =\mathcal{A}({}^S\0)/\mathcal{A}_{{}^LP}$.  Then $\bar{A}(\0) \simeq \bar{A}({}^S\0) \simeq \bar{A}_P \times \bar{A}_{{}^LP}$.
\end{proposition}

Here are some nice examples of mirror pairs. Consider the isotropic Grassmannian ${\rm IG}(k, 2n)$  and the associated Springer map $\nu_k: T^*{\rm IG}(k, 2n) \to \overline{\0}_k$.  The dual side  is the orthogonal Grassmannian ${\rm OG}(k,2n+1)$ with the dual Springer map  ${}^L \nu_k: T^*{\rm OG}(k, 2n+1) \to \overline{{}^S\0}_k$. Assume $k\leq \frac{2n}{3}$ in the following (the case $k >  \frac{2n}{3}$ will be discussed in Example \ref{e.kbig}).

The Springer map  ${}^L \nu_k$ is always birational and ${}^S \0_k$ is given by the partition $[3^k, 1^{2n-3k+1}]$.  It follows that $\pi_1({}^S \0_k) = A({}^S\0_k) =\mathbb{Z}_2$ and Lusztig's canonical quotient $\bar{A}({}^S\0_k)$ is trivial for $k$ even and $\mathbb{Z}_2$ for $k$ odd. Note that the orbit closure $\overline{{}^S \0_k}$ has minimal singularity $A_1$, so it is normal by \cite{KP82}.

When $k$ is even, the orbit $\0_k$ has partition $[3^k, 1^{2n-3k}]$, and the Springer map $\nu_k$ is birational.  In this case,  $A(\0_k) =\bar{A}(\0_k) $ is trivial. The orbit closure $\overline{\0}_k$ is non-normal if $2n>3k$ and normal if $2n=3k$ by \cite{KP82}.  When $k$ is odd, then $\0_k$ is given by $[3^{k-1},2^2,1^{2n-3k-1}]$ and the Springer map $\nu_k$ is of degree 2. In this case, we have $\pi_1(\0_k) = A(\0_k) =\bar{A}(\0_k)  =\mathbb{Z}_2$ and the closure is normal by \cite{KP82}.  It follows that if $k$ is even, then $(\widetilde{\0}_k, \overline{{}^S\0}_k)$ is a mirror pair, while for $k$ odd, the pair $(X_k, \overline{{}^S\0}_k)$ is a mirror pair, where $X_k \to \overline{\0}_k$ is the double cover. Note that when $k=1$, this is Example \ref{e.r=1}.

On the other hand, the double cover ${}^LX_k \to \overline{{}^S\0}_k$ is out of the scope of Theorem \ref{t.Richardson}. It is then natural to ask the following
\begin{question} \label{question}
Assume  $k\leq \frac{2n}{3}$ is odd.  Is the pair $(\widetilde{\0}_k, {}^LX_k)$  a mirror pair?
\end{question}

\subsection{Induced covers and Langlands duality}

For a parabolic subgroup $P \subset G$, we denote by $\p$ its Lie algebra. Let $\p = \mathfrak{l} \oplus \fu  $ be the Levi decomposition of $\p$. For a nilpotent orbit $\0_{\bf t}$ in $\fl$, Lusztig and Spaltenstein \cite{L-S} showed that $G \cdot (\fu+\overline{\0}_{\bf t})$ is a nilpotent orbit closure, say $\overline{\0}$. The variety $\fu+ \overline{\0}_{\bf t}$ is $P$--invariant and the surjective map
\begin{align*}
    \nu=\nu_{P,{\bf t}}: G \times^P(\fu+\overline{\0}_{\bf t}) \longrightarrow \overline{\0}	
\end{align*}
is generically finite and projective, which will be called {\em a generalized Springer map}. An {\em induced orbit} is a nilpotent orbit whose closure is the image of a generalized Springer map.  We denote an induced orbit by $\Ind_{\p}^{\g}(\cO_{{\bf t}})$, which depends only on the $G$--orbit of the pair $(\fl, \0_{\bf t})$, namely if $\p$ and $\p'$ are two parabolic subalgebras of $\g$ with the same Levi factor $\fl$, then $\Ind_{\p}^{\g}(\cO_{{\bf t}})=\Ind_{\p'}^{\g}(\cO_{{\bf t}})$ for any nilpotent orbit $\0_{\bf t}$ in $\fl$ (cf. \cite[Theorem 7.1.3.]{CM93}). A nilpotent orbit is called {\em rigid} if it is not induced from any parabolic subgroup. As shown by Lusztig (\cite{Lus79}), any nilpotent orbit induced from a special one is again special.

For an induced orbit $\0=\Ind_{\p}^{\g}(\cO_{\bf t})$, we denote by $\mu=\mu_{P, {\bf t}}: X_{P, {\bf t}} \to \overline{\0}$ the $G$-equivariant finite morphism obtained from $\nu$ by Stein factorization, which will be called the {\em induced cover} of $\0$ associated to the induction $\0=\Ind_{\p}^{\g}(\cO_{\bf t})$.

  Our second main theorem is an extension of Theorem \ref{t.Richardson} to special nilpotent orbits, which reads as follows

\begin{theorem}\label{t.induced}
\begin{itemize}
\item[(1)] If $\0 \subset \g$ is a rigid special orbit, so is its Springer dual ${}^S\0$.  Moreover, we have
\begin{align*}
    \pi_1(\0)=A(\0) = \bar{A}(\0) \simeq \bar{A}({}^S\0) =A({}^S\0) = \pi_1({}^S\0).	
\end{align*}
\item[(2)] Springer dual special orbits have isomorphic Lusztig's canonical quotient.
\item[(3)] Any non-rigid special orbit $\0$  is induced from a rigid special orbit, i.e., $\0 =\Ind_{\p}^{\g}(\cO_{{\bf t}})$ for some rigid special orbit $\cO_{{\bf t}}$ in $\fl$.  Taking the Langlands dual, we obtain the induction for ${}^S\0$ as the image of the generalized Springer map
\begin{align*}
	{}^L\nu: {}^LG \times^{{}^LP}({}^L\fu+\overline{{}^S\0}_{\bf t} ) \to \overline{{}^S\0}.
\end{align*}
\item[(4)] Consider the induced cover $\mu: X_{P, {\bf t}} \to \overline{\0}$ from a rigid special orbit $\0_{\bf t}$ and the dual induced cover
${}^L\mu: X_{{}^LP, {}^L{\bf t}} \to \overline{{}^S\0}$. Then we have
\begin{align*}
	\deg \mu \cdot \deg {}^L\mu =  |\bar{A}(\0)/\bar{A}(\0_{\bf t})| = |\bar{A}({}^S\0) / \bar{A}({}^S\0_{\bf t})|.
\end{align*}
\end{itemize}
\end{theorem}

As pointed out by E. Sommers, Claim (2) follows from \cite[Chapter 13]{Lu84}, and Claim (3) follows from \cite{L-S}. We will provide a combinatoric proof for Claim (2) as the proof from  \cite[Chapter 13]{Lu84}  is rather indirect.  For Claim (4), we will prove the following generalization of Proposition \ref{p.LusztigCanonicalQuotient}, which provides a better understanding of Lusztig's canonical quotient for induced orbits. Given an induction pair $(P, \cO_{\bf t})$, here $\cO_{\bf t}$ could be any special orbit, not necessarily rigid special. The induced cover $\mu: X_{P, {\bf t}} \to \overline{\0}$ corresponds to a finite subgroup $\mathcal{A}_{P, \cO_{\bf t}} \subset \mathcal{A}(\cO)$. Similar notations apply for the dual induction pair $({}^LP, {}^S\cO_{\bf t})$, i.e., $\mathcal{A}_{{}^LP, {}^S\cO_{\bf t}} \subset \mathcal{A}({}^S\cO)$.

\begin{proposition} \label{p.inducedCanonicalQuotient}
(i) $A(\cO_{\bf t})$ is a subgroup of $A(\cO)$, furthermore, $\bar{A}(\cO_{\bf t})$ is a subgroup of $\bar{A}(\cO)$. The same holds for Springer dual group ${}^S\cO_{\bf t}$ and ${}^S\cO$.

\noindent	(ii) $K(\cO) \subset \mathcal{A}_{P, \cO_{\bf t}}$ (resp. $K({}^S\cO) \subset \mathcal{A}_{{}^LP, {}^S\cO_{\bf t}} $), and $\bar{A}(\cO_{\bf t}) \subset \mathcal{A}_{P, \cO_{\bf t}}$ (resp. $\bar{A}({}^S\cO_{\bf t}) \subset \mathcal{A}_{{}^LP, {}^S\cO_{\bf t}}$).
	
\noindent	(iii) Let $\bar{A}_{P, \cO_{\bf t}}$ be the quotient $\mathcal{A}(\cO)/\mathcal{A}_{P, \cO_{\bf t}} $ and $\bar{A}_{{}^LP, {}^S\cO_{\bf t}} = \mathcal{A}({}^S\cO)/\mathcal{A}_{{}^LP, {}^S\cO_{\bf t}} $, then $\bar{A}(\cO) / \bar{A}(\cO_{\bf t}) \cong \bar{A}({}^S\cO)/ \bar{A}({}^S\cO_{\bf t}) \cong \bar{A}_{P, \cO_{\bf t}} \times \bar{A}_{{}^LP, {}^S\cO_{\bf t}} $.
\end{proposition}

Comparing Theorem \ref{t.induced} with Theorem \ref{t.Richardson}, what is missing for general special nilpotent orbits is the mirror symmetry part. How to formulate it in this situation? A natural idea is first to do it for rigid special orbits. Motivated by Theorem \ref{t.Richardson} and Theorem \ref{t.induced} (1), it is tempting to ask the following
\begin{question} \label{q.2}
Let $\0$ be a rigid special orbit and $\widehat{\0} \to \overline{\0}$ the finite map associated with its universal cover.  Similarly we have $ \widehat{{}^S\0} \to \overline{{}^S\0}$ on the dual side. Are $(\widehat{\0}, \widetilde{{}^S\0})$ and $(\widetilde{\0}, \widehat{{}^S\0})$ both mirror pairs?
\end{question}

\subsection{Asymmetry and mirror symmetry conjecture}

To approach Question \ref{question}, we take  a closer look of the degrees $(\deg \mu_P, \deg \mu_{{}^LP})$  in Theorem \ref{t.Richardson}.
It turns out that they form a continuous sequence. Namely, their value set is equal to the following (for some $\alpha,\beta,m$ determined by the partitions, see Proposition \ref{degree-range} for details):
\begin{align*}
	\{(2^\beta, 2^{\alpha+m}), (2^{\beta+1}, 2^{\alpha+m-1}), \cdots, (2^{\beta+m}, 2^\alpha)\}.
\end{align*}
We call this sequence the \emph{footprint} for the Springer dual pair $(\cO, {}^S\cO)$. It is surprising that the footprint is not symmetric in general, i.e., $\alpha \neq \beta$. One may wonder if mirror symmetry still holds outside the footprint. Question \ref{question} provides the first example to check.

We will consider more generally special spherical orbits in $B_n$ and $C_n$, which are Springer dual and  given by the following from \cite{Pan94} (for $r, l \geq 1$):
$$
\0_C:=\0_{r, l}^C := \0_{[2^{2r},1^{2l}]} \subset \mathfrak{sp}(4r+2l), \quad \quad  \0_B:=\0_{r, l}^B = \0_{[3,2^{2r-2},1^{2l+2}]} \subset \mathfrak{sp}(4r+2l+1).
$$

 It follows from \cite{KP82} that both closures $\overline{\0}_C $ and $\overline{\0}_B$  are normal. A nice observation (Proposition \ref{p.levy}) due to Paul Levy is that the double cover of $\0_B$ is given by
 $$
\0_D := \0_{r, l}^D := \0_{[2^{2r},1^{2l+2}]} \subset \mathfrak{so}(4r+2l+2).
$$

Let $G_C := {\rm Sp}_{2n}$ and $G_D := {\rm SO}_{2n+2}$, then we have the Jacobson-Morosov resolutions for nilpotent orbit closures:
\begin{align*}
	 G_C \times^{P_C} \n_C  \longrightarrow \overline{\cO}_C, \quad
	 G_D \times^{P_D} \n_D  \longrightarrow \overline{\cO}_D.
\end{align*}

As observed in \cite[Table 1]{BP19}, we have $\n_C \simeq {\rm Sym}^2 V_C$ and $\n_D \simeq \wedge^2 V_D$ for some $2r$--dimensional vector spaces $V_C$ and $V_D$.  Now take successive blowups along strict transforms of stratums defined by ranks
$$\widehat{\n}_C \to \n_C, \ \text{and} \ \widehat{\n}_D \to \n_D,$$
which are just affine versions of varieties of complete quadrics and complete skew forms (\cite{Be97, Th99}).
This gives log resolutions of $\overline{\0}_C$ and $\overline{\0}_D$:
\begin{align*}
	 G_C \times^{P_C} \widehat{\n}_C \longrightarrow \overline{\cO}_C, \quad
	 G_D \times^{P_D} \widehat{\n}_D  \longrightarrow \overline{\cO}_D.
\end{align*}
 This allows us to compute stringy E-functions for $\overline{\0}_C$ and $\overline{\0}_D$, which gives negative answers to both Question \ref{question}  and Question \ref{q.2} as follows:

\begin{proposition} \label{p.negative}
(i) When $r=1$, the orbits $\0_C$ and $\0_B$ are both Richardson. The universal cover of $\overline{\0}_C$ is a mirror of $\overline{\0}_B$, but $\overline{\0}_C$ is not a mirror of the universal cover of $\overline{\0}_B$.

(ii) When $r=2$, the orbits $\0_C$ and $\0_B$ are both rigid special.  The orbit closure $\overline{\0}_C$ is not a mirror of the universal cover of $\overline{\0}_B$ when $l=1$.
\end{proposition}

Motivated by these computations and the asymmetry of footprint, we  can now formulate the mirror symmetry conjecture for special nilpotent orbits as follows:
\begin{conjecture} \label{conj}
\begin{itemize}
\item[(i)]Let $\0_C$ be a rigid special orbit in $\mathfrak{sp}_{2n}$ and   $\0_B \subset \mathfrak{so}_{2n+1}$  its Springer dual. Let $\widehat{\0}_C \to \overline{\0}_C$ be the  universal cover. Then $(\widehat{\0}_C, \widetilde{\0}_B)$ is a mirror pair,  in the sense that they have the same stringy E-function.
\item[(ii)] For a special non-rigid orbit $\0 =\Ind_{\p}^{\g}(\cO_{{\bf t}})$ in $\mathfrak{sp}_{2n}$ induced from a rigid special orbit $\0_{\bf t}$, let $\widehat{\0_{\bf t}} \to \overline{\cO}_{\bf t}$ be the universal cover. Then we have  a  generically finite map
\begin{align*}
	\hat{\nu}=\widehat{\nu}_{P,{\bf t}}: G \times^P(\fu+\widehat{\0_{\bf t}} ) \to \overline{\0}.
\end{align*}
Let $\widehat{X}_{P, {\bf t}} $  be the finite cover of $\overline{\0}$ obtained by Stein factorization of $\hat{\nu}$.
Then  the pair $(\widehat{X}_{P, {\bf t}},  X_{{}^LP, {}^L{\bf t}})$ is a mirror pair, in the sense that they have the same stringy E-function.
\end{itemize}
\end{conjecture}

Here is the organization of the paper. We first recall combinatorics of nilpotent orbits in Section 2, where we prove that rigid special orbits are preserved by the Springer dual map, and Springer dual orbits have the same Lusztig's canonical quotient.  Section 3 is devoted to the proof of Theorem \ref{t.Richardson} and Proposition \ref{p.LusztigCanonicalQuotient}, where we also find some interesting examples of mirror pairs beyond parabolic covers of Richardson orbits. Section 4 is devoted to proving Theorem \ref{t.induced}. In the final section, we first study the asymmetry of footprints of Richardson orbits. Then we compute stringy E-functions for some special spherical nilpotent orbits, which proves Proposition \ref{p.negative}.

Notations: throughout this paper, the symbol $\equiv$ means congruence modulo 2, and $\mathbb{Z}_2 = \mathbb{Z}/2\mathbb{Z}$ is the elementary 2-group with one generator.

\section*{Acknowledgements}
This project was initiated when the third author visited the IASM (Institute for Advanced Study in Mathematics) of Zhejiang University in 2021. Both the second and third authors would like to thank the IASM for the wonderful research environment.  We are grateful to E. Sommers for helpful communications and suggestions to a previous version.   Y. R. would like to thank Gukov for a series of beautiful lectures in the IASM Geometry and Physics seminar and many valuable discussions afterward. Y. W. would like to thank Weiqiang He, Xiaoyu Su, Bin Wang, and Xueqing Wen for their helpful discussions and would also like to congratulate Prof. Choi for his support. B. F. is supported by the National Natural Science Foundation of China (No. 12288201).  Y. W. is supported by a KIAS Individual Grant (MG083901) at Korea Institute for Advanced Study and a POSCO Science fellowship.

\section{Nilpotent orbits and partitions}

\subsection{Special partitions}

Fix an integer $N$. We denote by $\cP(N)$ the set of partitions $\mathbf{d} =[d_1 \geq d_2 \cdots \geq d_N]$ of $N$ (here we allow 0 in the parts).
For $\varepsilon= \pm 1$, we define
\begin{align}
	\mathcal{P}_{\varepsilon}(N) = \big\{ [d_1, \ldots, d_N] \in \mathcal{P}_{}(N) \ \big{|} \ \sharp \{j \ | \ d_j =i \} \ \text{is even for all $i$ with} \ (-1)^i= \varepsilon      \big\}.
\end{align}

It is well-known that $\cP(N)$ (resp. $\cP_{1}(2n+1)$, $\cP_{-1}(2n)$) parameterizes nilpotent orbits in $\mathfrak{sl}_N$ (resp. $\mathfrak{so}_{2n+1}, \mathfrak{sp}_{2n}$). In other words, nilpotent orbits in $\mathfrak{so}_{2n+1}$ (resp. $\mathfrak{sp}_{2n}$) are in one-to-one correspondence with the set of partitions of $2n+1$ (resp. $2n$) in which even parts (resp. odd parts) occur with even multiplicity.	

The set $\cP(N)$ is partially ordered as follows:  $ \mathbf{d} =\left[ d_1, \ldots, d_N  \right] \geq \mathbf{f}= \left[ f_1, \ldots, f_N  \right] $ if and only if $ \sum_{j=1}^k d_j \geq \sum_{j=1}^k f_j, \ \text{for all} \ 1 \leq k \leq N.	$ This induces a partial order on $\mathcal{P}_{\varepsilon}(N)$, which coincides with the partial ordering on nilpotent orbits given by inclusion of closures(cf. \cite[Theorem 6.2.5]{CM93}).

Given a partition $\mathbf{d}=\left[ d_1, \ldots, d_N  \right]$, its dual partition $\mathbf{d}^*$ is defined by $d^*_i = \sharp \left\{ j \ \big{|} \ d_j \geq i \right\} $ for all $i$. We call a partition $\mathbf{d} \in \mathcal{P}_{\varepsilon}(N)$ {\em special} if its dual  partition $\mathbf{d}^*$ lies in $\mathcal{P}_{\varepsilon}(N)$, and the corresponding nilpotent orbits are called special orbits (cf. \cite[Proposition 6.3.7]{CM93}).  We denote by $\mathcal{P}_{\varepsilon}^{sp}(N)$ the set of all special partitions in $\mathcal{P}_{\varepsilon}(N)$.

 Here is a simple criterion for special partitions:  a partition in $\cP_1(2n+1)$  is special if and only if it has an even number of odd parts between any two consecutive even parts and an odd number of odd parts greater than the largest even part, and a partition in $\cP_{-1}(2n)$ is special if and only if it has only even parts or it has an even number of even parts between any two consecutive odd ones and an even number of even parts greater than the largest odd part. We can illustrate these partitions as follows (throughout this paper, $\equiv$ means mod 2):
\begin{align*}
B_n: 	[d_1 \equiv d_2 \equiv \cdots \equiv d_{2l-1} \equiv 1,  d_{2l} \equiv d_{2l+1}\equiv 0, \ldots, d_{2j} \equiv d_{2j+1}, \cdots, d_{2n} \equiv d_{2n+1}].	
\end{align*}
\begin{align*}
C_n: &	[d_1 \equiv d_2 \equiv \cdots \equiv d_{2k} \equiv 0 ,  d_{2k+1} \equiv d_{2k+2} \equiv 1, \ldots , d_{2j-1} \equiv d_{2j}, \cdots,  d_{2n-1}\equiv  d_{2n}], \\
     & \text{or} \ d_i \equiv 0, \forall i.
\end{align*}

\subsection{Springer dual partitions}

For a partition $\mathbf{d}=\left[ d_1,\ldots, d_{2n+1} \right]$ in $\cP(2n+1)$, there is a unique largest partition in $\cP_{1}(2n +1)$ dominated by $\mathbf{d}$. This partition, called the $B$--collapse of $\mathbf{d}$ and denoted by $\mathbf{d}_B$, may be obtained as follows: if $\mathbf{d}$ is not already in $\cP_{1}(2n+ 1)$, then at least one of its even parts must occur with odd multiplicity. Let $q$ be the largest such part. Replace the last occurrence of $q$ in $\mathbf{d}$ by $q-1$ and the first subsequent part $r$ strictly less than $q-1$ by $r+1$. We call this procedure a collapse. Repeat this process until a partition in $\cP_{1}(2n +1)$ is obtained, which is denoted by $\mathbf{d}_B$

Similarly, there is a unique largest partition $\mathbf{d}_C$ in $\cP_{-1}(2n)$ dominated by any given partition $\mathbf{d}$ of $2n$, which is called the $C$--collapse of $\mathbf{d}$. The procedure to obtain $\mathbf{d}_C$ from $\mathbf{d}$ is similar to that of $\mathbf{d}_B$.  	

Recall that  Langlands dual groups $G$ and ${}^LG$ have the same Weyl group $\mathcal{W}$, hence for any special orbit $\0$ in $\g$, there exists a unique special orbit ${}^S\0 \subset {}^L\g$ such that $(\0, 1)$ and $({}^S\0, 1)$ correspond to the same irreducible representation of $\mathcal{W}$ via the Springer correspondence. This gives a bijection between special nilpotent orbits in $\g$ and ${}^L\g$, which is called the Springer dual map. We can now describe this map in terms of partitions.

Given a partition $\mathbf{d} = [d_1, \cdots, d_k]$ (with $d_k \geq 1$). We define $\mathbf{d}^- = [d_1, \cdots, d_k-1]$ and $\mathbf{d}^+ = [d_1+1, \cdots, d_k]$. The following result is from \cite[Chapter III]{Spa82} (see also \cite[Proposition 4.3]{KP}).
\begin{proposition} \label{p.Springerdual}
The map $\mathbf{d} \mapsto (\mathbf{d}^+)_B$ gives an bijection $S: \cP^{sp}_{-1}(2n)\to \cP^{sp}_{1}(2n +1)$, which is the Springer dual map. Its inverse is given by $\mathbf{f} \mapsto (\mathbf{f}^-)_C$.
\end{proposition}

It follows that if $\mathbf{d}$ is the partition of a special orbit $\cO$, then the Springer dual nilpotent orbit ${}^S\cO$ in ${}^L\g$ is given by the partition  $S(\mathbf{d})$.  We end this subsection with the following trick for collapses.

\begin{lemma} \label{l.trick}
	Let $\mathbf{d}=[d_1 \equiv  d_2 \equiv \dots \equiv d_{2k} \equiv 1, d_{2k+1}\equiv 0 ] \in \mathcal{P}(2n) $, then
	\begin{align*}
		((\mathbf{d}_C)^+)_B =  [d_1, \dots, d_{2k}, d_{2k+1}+1 ].
	\end{align*}
\end{lemma}
\begin{proof}
Recall that the collapse of type $C$ happens when $d_{2i-1} \neq d_{2i}$ for $1 \leq i \leq k$, and it makes the following changes:
\begin{align*}
	d_{2i-1} \mapsto d_{2i-1} -1, \quad d_{2i} \mapsto d_{2i} +1.
\end{align*}
Let $I$ be the subset of $\{1, \dots, k \}$, such that, for $i \in I$, $d_{2i-1} \neq d_{2i}$. So for $j \in I^{\complement}:= \{1, \dots, k \} \backslash I$, $d_{2j-1}=d_{2j}$. Then
\begin{align*}
	\mathbf{d}_C = [d_1^\prime, \dots, d_{2k}^\prime, d_{2k+1}],
\end{align*}
where
\begin{align*}
	& \{d_{2l-1}^\prime, d_{2l}^\prime \} = \{d_{2l-1}, d_{2l} \}, \quad \text{if} \quad l \in I^\complement,\\
	& \{d_{2l-1}^\prime, d_{2l}^\prime \} = \{d_{2l-1}-1, d_{2l}+1 \}, \quad \text{if} \quad l \in I.
\end{align*}
By definition
\begin{align*}
    (\mathbf{d}_C)^+ = [d_1^\prime+1, d_2^\prime, \dots, d_{2k}^\prime, d_{2k+1}].
\end{align*}
Then the proof is divided into two cases: $1 \in I$ or $1 \in I^\complement$.

\bf{Case 1:} if $1 \in I^\complement$, set $m+1 = {\rm min} \{i | i\in I\}$, then 
\begin{align*}
	(\mathbf{d}_C)^+ = [d_1 + 1, d_2 \equiv \dots \equiv d_{2m} \equiv 1, d_{2m+1}^\prime \equiv 0,  \dots ].
\end{align*}
Note that $d^\prime_{2m+1} = d_{2m+1}-1$ and $d^\prime_{2m+2} = d_{2m+2}+1$.
By the rule of $B$--collapse, we replace $d_1+1$ by $d_1$ and replace the first subsequent part $d_{2i+1}$ strictly less than $d_{1}-1$ by $d_{2i+1}+1$. If $d_{2i+1} \equiv 1 $, then we repeat the above procedure, thus we have
\begin{align*}
	[d_1 + 1, \dots, d_{2m}, d_{2m+1}^\prime,  \dots ]_B =[d_1, \dots, d_{2m}, d_{2m+1}^\prime+1,  \dots ]_B.
\end{align*}
Note that $d_{2m+1}^\prime+1=d_{2m+1}$ and $d_{2m+2}^\prime+1=d_{2m+2}+1$ .
Hence the situation is to show
\begin{align*}
	[d_{2m+1}^\prime+1,  \dots, d_{2k}^\prime, d_{2k+1}]_B = [d_{2m+1},  \dots, d_{2k}, d_{2k+1}+1],
\end{align*}
which is reduced to the second case.

\bf{Case 2:} if $1 \in I$, without loss of generality, suppose
\begin{align*}
	(\mathbf{d}_C)^+ = [d_1^\prime +1, d_2^\prime , \dots, d_{2m+1}^\prime \equiv 0,  \dots ].
\end{align*}
Since $1 \in I$, then 	$d_1^\prime +1=d_1$ and $d_2^\prime = d_2 +1$. Similarly, $B$--collapse replace $d_2+1$ with $d_2$ and replace the first subsequent part $d_{2i+1}^\prime$ strictly less than $d_{2}-1$ by $d_{2i+1}^\prime+1$. If $d_{2i+1}^\prime \equiv 1 $, then $d_{2i+1}^\prime=d_{2i+1}$ and we are back to \bf{Case 1}. If $d_{2i+1}^\prime \equiv 0$, note that $d_{2i+1}^\prime=d_{2i+1}-1$ and $d_{2i+2}^\prime=d_{2i+2}+1$. Then
\begin{align*}
	d_{2i+1}^\prime \mapsto d_{2i+1}^\prime+1 = d_{2i+1}, \quad d_{2i+2}^\prime \mapsto d_{2i+2}^\prime-1=d_{2i+2},
\end{align*}
and again replace the first subsequent part $d_{2j+1}^\prime$ strictly less than $d_{2i+2}-1$ by $d_{2j+1}^\prime+1$. Repeat the above procedures, we have
\begin{align*}
	[d_1^\prime + 1, \dots, d_{2m}^\prime, d_{2m+1}^\prime,  \dots ]_B =[d_1, \dots, d_{2m}, d_{2m+1}^\prime+1,  \dots ]_B.
\end{align*}

Now the situation is to show
\begin{align*}
	[d_{2m+1}^\prime+1,  \dots, d_{2k}^\prime, d_{2k+1}]_B = [d_{2m+1},  \dots, d_{2k}, d_{2k+1}+1],
\end{align*}
which goes back to the first case.  We conclude the proof by induction.
\end{proof}

\subsection{Induced orbits and Richardson orbits}

Let $\g$ be a Lie algebra of type $B_n$ or $C_n$. We denote by $\langle \cdot, \cdot \rangle_{\varepsilon} $, $\varepsilon=1$ (resp. $-1$), the bilinear form on $\bC^N$, $N=2n+1$ (resp. $2n$) with respect to the type of Lie algebra $\g$. Let $\p$ be a parabolic subalgebra of $\g$. There exists an admissible isotropic filtration of $\bC^N$, which is preserved by $\p$:
\begin{align*}
	\{0\}=F_0 \subset F_1 \dots \subset F_{l-1} \subset  F_l \subset F_l^{\bot} \subset F_{l-1}^{\bot} \subset \dots \subset F_1^{\bot} \subset F_0^{\bot} = \bC^N,
\end{align*}
here $F_i^{\bot}$ means the orthogonal complement of $F_i$ with respect to the bilinear form $ \langle \cdot, \cdot \rangle_{\varepsilon} $. Denote by
\begin{align*}
	V_i = F_i/F_{i-1}, \  V_j^\prime= F_{j-1}^{\bot}/F_{j}^{\bot} \quad \text{for} \ 1 \leq i, j \leq l, \quad \text{and} \quad W= F_{l}^{\bot}/F_l.
\end{align*}
The Levi subalgebra $\mathfrak{l}$ of $\p$ preserves the following decomposition:
\begin{align}
\bC^N = V_1 \oplus V_2 \oplus \ldots \oplus V_l \oplus W \oplus V_l^{\prime} \oplus \ldots \oplus V_1^{\prime}, \nonumber
\end{align}
here $ \dim V_i = \dim V_i^{\prime} $ and $W$ is a quadratic space of type $\varepsilon$. Let $ p_i = \dim V_i $ and $ q = \dim W$, then we have
\begin{align}
 \mathfrak{l} \cong \mathfrak{gl}_{p_1} \times \ldots \times \mathfrak{gl}_{p_l} \times \g^{\prime},	\label{levi-type}
\end{align}
where $\g^{\prime}$ has the same type as $\g$. We say $\mathfrak{l}$ is of type $(p_1, \ldots, p_l; q)$. Note that $2\sum_{k=1}^l p_k +q = N $.

Given a nilpotent orbit in $\mathfrak{l}$, say $ \cO_{\vec { \mathbf d } }$ for ${\vec { \mathbf d } }=  \{ \mathbf d_1, \ldots \mathbf d_l ; \mathbf d_0 \} $, where each $\mathbf d_i= \left[ d_{i1}, \ldots, d_{i k_i} \right], \ i=0, \ldots, l$ corresponds to a partition in each component.
Consider the new partitions
\begin{align}
\begin{split}
    \mathbf d^{\prime} &= \left[ d_1^{\prime}, \ldots, d_r^{\prime} \right],	\quad \text{with} \quad d_j^{\prime} = 2 \sum_{i=1}^l d_{ij} + d_{0j}, \label{induction} \\
    \textbf{Even} &= [d_1^{\prime \prime}, \dots, d_l^{\prime \prime}], \quad \text{with} \quad d_j^{\prime \prime} = 2 \sum_{i=1}^l d_{ij}.
\end{split}
\end{align}
Here we use the notation that if the index $j$ exceeds the range of the indexes of $\mathbf d_i$'s, we take $d_{ij}=0$. Let $\mathbf d^{\prime}_\varepsilon$ be the collapse of $\mathbf d^{\prime}$ with respect to Lie algebra type. Then the partition of the induced orbit is given by $\mathbf d^{\prime}_\varepsilon$, i.e.
\begin{align}
 \Ind_{\fl}^{\g} \left( \cO_{\vec { \mathbf d } }   \right)	= \cO_{\mathbf d^{\prime}_{\varepsilon}}, \nonumber
\end{align}

Recall that Richardson orbits are orbits induced from the zero orbit (namely $\cO_{\vec{\mathbf{d}}} = \cO_{\vec{1}} :=0$), hence the partitions  $\mathbf{d}_i, 0 \leq i \leq l $ consist of 1's.  It follows that the new partition $\mathbf{d}^{\prime}$ has the property that there exists a line dividing $\mathbf{d}^{\prime}= \{ d_1^{\prime}, \ldots, d_r^{\prime} \}$ into two parts, the numbers in the left part  are all odd, and the numbers in the right part are all even, i.e.
\begin{align}
\mathbf{d}^{\prime}=[ d_{1}^{\prime} \equiv \ldots \equiv d_{k}^{\prime} \equiv 1 ,  d_{k+1}^{\prime} \equiv \ldots \equiv d_{r}^{\prime} \equiv 0 ].	\nonumber
\end{align}
We call $\mathbf{d}^{\prime}$ the partition induced from the parabolic subgroup. 
Using this observation, one can determine which nilpotent orbits are Richardson. Spaltenstein was the first one who calculated polarizations of Richardson orbits (but unpublished). Hesselink gave a proof and a description in \cite{He78}, and Kempken gave another explanation in \cite{Ke81}. The description of Richardson orbits in $B_n$ or $C_n$ is as follows:

For $B_n$, let $\mathbf{d}= \left[ d_1, \ldots, d_{2r+1} \right] \in \cP_{1}(2n+1) $.  The  orbit $\cO_{\mathbf{d}}$ is Richardson if and only if  $\min_{j \geq 1} \{j : d_j \ \rm{is \ even}  \}  $ is even, say $2l$,  and  for all $ j \geq l $,  we have $d_{2j} \equiv d_{2j+1} $, and if $d_{2j-1}$, $d_{2j}$ are odd, then $d_{2j-1} \geq d_{2j} +2 $.

For $C_n$, let $\mathbf{d}= \left[ d_1, \ldots, d_{2r+1} \right] \in \cP_{-1}(2n) $. The  orbit $\cO_{\mathbf{d}}$ is Richardson if and only if $ \max_{j \geq 1} \{j : d_j \ \rm{is \ odd}  \}  $ is even, say $2k \geq 0$, and if $k \neq 0$ we have for all $ 1 \leq j \leq k $, $d_{2j-1} \equiv d_{2j} $, and if $d_{2j}$, $d_{2j+1}$ are even, then $d_{2j} \geq d_{2j+1} +2 $ 	

We can illustrate these partitions as follows:
 \begin{align*}
B_n: 	[d_1 \equiv d_2 \equiv \cdots \equiv d_{2l-1} \equiv 1 ,  \underbrace{d_{2l} \equiv d_{2l+1}\equiv 0, \ldots, d_{2j} \equiv d_{2j+1}, \cdots, d_{2n+1}}_{\text{if}\  d_{2j-1} \equiv d_{2j} \equiv 1, \text{ then} \  d_{2j-1} \geq d_{2j} +2}].
\end{align*}

  \begin{align*}
C_n: 	[\underbrace{d_1\equiv d_2, \cdots, d_{2k-1} \equiv  d_{2k}\equiv 1}_{\text{if}\  d_{2j} \equiv d_{2j+1} \equiv 0, \text{ then} \  d_{2j} \geq d_{2j+1} +2}, d_{2k+1} \equiv \ldots \equiv  d_{2n}\equiv 0 ].
\end{align*}

\subsection{Rigid special orbits under Springer duality}

Recall that a nilpotent orbit is said rigid if it is not induced from any parabolic subgroup.
The partitions in $\cP_\varepsilon(N)$ corresponding to rigid orbits are given by  $\mathbf{d}=[d_1, \cdots, d_N]$
such that (cf.  \cite[p.117]{CM93})
$$
0 \leq d_{i+1} \leq d_i \leq d_{i+1}+1 \quad \text{and} \quad \sharp\{j|d_j=i\} \neq 2 \quad \text{if} \quad  (-1)^{i+1} = \varepsilon.
$$
In other words, it has the form $[k^{m_k}, (k-1)^{m_{k-1}}, \cdots, 1^{m_1}]$ such that $m_j \geq 1$ and $m_j \neq 2$ when $j$ is odd for $B_n$ (resp. even for $C_n$).

Let $\cN^{rs}$ and ${}^L\cN^{rs}$ be the sets of rigid special orbits in $\mathfrak{sp}_{2n}$ and $\mathfrak{so}_{2n+1}$ respectively. It follows from previous discussions that the partitions of the rigid special orbits in $\mathfrak{so}_{2n+1}$ and $\mathfrak{sp}_{2n}$ have the following forms
		\begin{align*}
		\text{For} \quad  B_n=\mathfrak{so}_{2n+1}:  \mathbf{d}_B= \left[ k^{2n_k-1}, (k-1)^{2n_{k-1}}, \cdots, 2^{2n_2}, 1^{2n_1} \right], 		\end{align*}
		where $k$ is odd and if $i (\neq k)$ is odd then $n_i \geq 2$, otherwise, $n_i \geq 1$.
		\begin{align*}
		\text{For} \quad 	C_n=\mathfrak{sp}_{2n}: \mathbf{d}_C= \left[ k^{2n_k}, (k-1)^{2n_{k-1}}, \cdots, 2^{2n_2}, 1^{2n_1} \right],
		\end{align*}
		we have $ n_{\text{even}} \geq 2$ and $n_{\text{odd}} \geq 1$.

\begin{lemma} \label{rigid-special-under-Springer-dual}
The Springer dual map $S: \cN^{sp} \to {}^L\cN^{sp}$ preserves rigid special orbits, namely
	\begin{align*}
		S(\cN^{rs}) = {}^L\cN^{rs}.
	\end{align*}
\end{lemma}
\begin{proof}
	Given a rigid special orbit $\cO_{\bf d_C} \subset \mathfrak{sp}_{2n}$, where
	\begin{align*}
			\mathbf{d}_C= \left[ k^{2n_k}, (k-1)^{2n_{k-1}}, \cdots, 2^{2n_2}, 1^{2n_1} \right].
	\end{align*}
	By Proposition \ref{p.Springerdual}, the Springer dual $\cO_{\bf d_B} = {}^S\cO_{\bf d_C}$ is given by the $B$--collapse:
	\begin{align*}
		\mathbf{d}_B= \left[ k+1, k^{2n_k-1}, (k-1)^{2n_{k-1}}, \cdots, 2^{2n_2}, 1^{2n_1} \right]_B.
	\end{align*}
	If $k$ is even, then
	\begin{align*}
		\mathbf{d}_B= \left[k+1, k^{2n_{k}-2}, (k-1)^{2n_{k-1}+2}, \cdots, 2^{2n_2-2}, 1^{2n_1+2} \right].
	\end{align*}
	If $k$ is odd, then
	\begin{align*}
		\mathbf{d}_B= \left[k^{2n_{k}+1}, (k-1)^{2n_{k-1}-2}, \cdots, 2^{2n_2-2}, 1^{2n_1+2} \right].
	\end{align*}
	That is, the $B$--collapse follows the rule
	\begin{align*}
		l^{2n_{l}} \mapsto \begin{cases}(l+1) l^{2n_{l}-2}(l-1), & l \text { even } \\ l^{2n_{l}}, & l \text { odd }\end{cases}.
	\end{align*}
	This shows that $\cO_{\bf d_B}$ is a rigid special orbit of type $B_n$.
\end{proof}

\subsection{Lusztig's canonical quotient} \label{LCQ}

Recall that an orbit $\0_{\mathbf d}$ is said {\em rather odd} if every odd part of $\mathbf{d}$ has multiplicity 1,
and {\em very even} if every even part has even multiplicity.  We define the following two integers for a given partition $\mathbf{d}$:
$$a_{odd}=\sharp \{ \text{distinct odd} \ d_i\}, \quad a_{even}=\sharp \{ \text{distinct non-zero even} \ d_i\}.$$

Following \cite[Corollary 6.1.6]{CM93}), the component group ($A(\cO)=G^{\rm ad}_x/(G^{\rm ad}_x)^{\circ}$ for $x \in \cO$), and the fundamental group ($\pi_1(\cO)= G^{\rm sc}_x/(G^{\rm sc}_x)^{\circ}$ with $G^{\rm sc}$ the universal cover of $G$) of an orbit
$\0$ associated to a  partition $\mathbf{d}=[d_1, \cdots, d_N]$  are determined as follows:
\begin{center}
\begin{tabular}{|c|c|c|}
 \hline
 type &     $A(\0) $  & $\pi_1(\0)$ \\
  \hline
    &    &  $1\to \mathbb{Z}_2 \to \pi_1(\0) \to \mathbb{Z}_2^{a_{odd}-1}\to 1$ if $\0$ is rather odd
  \\ $B_n$  & $\mathbb{Z}_2^{a_{odd}-1}$ & or  \\
  &  &   \   $\mathbb{Z}_2^{a_{odd}-1}$ otherwise\\
  \hline
   &     $\mathbb{Z}_2^{a_{even}}$ for very even orbits  &
   \\
 $C_n$ & or  & $ \mathbb{Z}_2^{a_{even}}$ \\
 & $\mathbb{Z}_2^{a_{even}-1}$ otherwise & \\
  \hline
\end{tabular}
\end{center}

Since we use the notations $G={\rm Sp}_{2n}$ and ${}^LG= {\rm SO}_{2n+1}$, we define
\begin{align*}
	\mathcal{A}(\cO):=  \left\{
\begin{array}{rcl}
& G_x/(G_x)^{\circ}, \quad \text{for} \quad x \in \cO     & \text{in type $C_n$}, \\
& {}^LG_{x}/({}^LG_x)^{\circ}, \quad \text{for} \quad x \in \cO     & \text{in type $B_n$}.
\end{array}
\right.
\end{align*}
Then, in type $C_n$, $\mathcal{A}(\cO)=\pi_1(\cO)$ and in type $B_n$, $\mathcal{A}(\cO)=A(\cO)$. Consider a special nilpotent orbit $\cO_{\mathbf{d}}$ in $B_n$ or $C_n$. We have a better description of $\mathcal{A}(\cO)$ following \cite{So, He78}. Let $\mathbf{d}_{\varepsilon}=[d_1, d_2, \dots, d_{N}]$ be a special partition corresponding to the special orbit $\cO$ of type $B_n$ or $C_n$. Let $\varepsilon=(-1)^\chi$, $\chi = 0, 1$ and $d_j=0$ if $j > N$. Consider the following set:
\begin{align*}
	B(\mathbf{d}_{\varepsilon})= \{ d_j \in \mathbb{N} \mid d_j > d_{j+1}, d_j \notequiv \chi \}.
\end{align*}
Consider the elementary 2-group with basis consisting of elements $x_j$, $j \in \mathbb{N}$. Then $\mathcal{A}(\cO)$ is the subgroup indexed by $B(\mathbf{d}_{\varepsilon})$ \cite[Section 5.3]{He78} :
\begin{itemize}
	\item[--]If $\chi=1$, i.e., $\cO$ is type $C_n$, then $\mathcal{A}(\cO)=\{x= x_{i_1} + x_{i_2} + \dots + x_{i_m} \mid i_j \in B(\mathbf{d}_{\varepsilon}), 1 \leq j \leq m \}$,
	\item[--]If $\chi=0$, i.e., $\cO$ is type $B_n$, then $\mathcal{A}(\cO)=\{ x= x_{i_1} + x_{i_2} + \dots + x_{i_{m}},\ m \equiv 0 \mid i_j \in B(\mathbf{d}_{\varepsilon}), 1 \leq j \leq m  \}$.
\end{itemize}
If there is no misunderstanding, by abuse of notation, we sometimes use $d_i \in \mathbf{d}_{\varepsilon}$ to denote $x_{d_i}$ when we describe $\mathcal{A}(\cO)$.

Being a quotient of $\mathcal{A}(\0_{\mathbf{d}})$,  Lusztig's canonical quotient $\widebar{A} (\cO_{\mathbf{d}})$ is isomorphic to $\mathbb{Z}_2^{q}$ for some $q$. We will give a more detailed explanation of the component group and Lusztig's canonical quotient group in Section \ref{MSPCRO}. Here is the way to compute Lusztig's canonical quotient following \cite{So} (see also \cite[Proposition 2.2]{Won17}).
\begin{proposition}\label{p.CombCanonicalQuotient} \
Lusztig's canonical quotient $\bar{A}(\0)$ can be determined by the following procedure:
	\begin{itemize}
		\item \textbf{\emph{Type} $B_n$:} Let $\mathbf{d}_B=[d_1, \cdots, d_{2r+1}]$ with $d_{2r+1} \geq 1$. We separate all even parts of the form $d_{2j}=d_{2j+1}=\alpha$ and odd parts of the form $d_{2j+1}=d_{2j+2}=\beta $, then we  obtain the remaining $2q+1$ parts:
		\begin{align*}
			\mathbf{d}_B=[d^{\prime \prime}_{1}, \cdots, d^{\prime \prime}_{2q+1}] \cup [\alpha_1, \alpha_1, \cdots, \alpha_x, \alpha_x] \cup [\beta_1, \beta_1, \cdots, \beta_y, \beta_y].
		\end{align*}
		In this case, we have $\widebar{A}(\cO_{\mathbf{d}_B})= \mathbb{Z}_2^{q}$. 
		\item \textbf{\emph{Type} $C_n$:} Let $\mathbf{d}_C=[d_1, \cdots, d_{2r+1}]$, where $d_{2r} \geq 1$ and  $d_{2r+1}$ may be 0. We separate all odd parts of the form $d_{2j+1}=d_{2j+2}=\alpha$ and even parts of the form $d_{2j}=d_{2j+1}=\beta $, then we obtain the remaining $2q+1$ parts:
		\begin{align*}
			\mathbf{d}_C=[d^{\prime \prime}_{1}, \cdots, d^{\prime \prime}_{2q+1}] \cup [\alpha_1, \alpha_1, \cdots, \alpha_x, \alpha_x] \cup [\beta_1, \beta_1, \cdots, \beta_y, \beta_y].
		\end{align*}
		In this case, we have $\widebar{A}(\cO_{\mathbf{d}_C})=\mathbb{Z}_2^{q}$.
	\end{itemize}
\end{proposition}
Let $r_i = \sharp \{j \mid d_j =i \}$ and
\begin{align*}
	S_{\text{odd}} = \{ i \in \mathbb{N} \mid i \notequiv \chi, r_i \equiv 1 \}, \quad
	S_{\text{even}} = \{ i \in \mathbb{N} \mid i \notequiv \chi, r_i \equiv 0 \ \text{and} \ r_i \neq 0 \}.	
\end{align*}

\begin{remark} \label{r.CombCanonicalQuotient}
	From the proof of \cite[Proposition 2.2]{Won17}, the generators of $\bar{A}(\cO)$ are decribed in terms of $[d^{\prime \prime}_{1}, \cdots, d^{\prime \prime}_{2q+1}]$ as follows:
	\begin{itemize}
		\item[--] In type $B_n$, if $d^{\prime \prime}_{2j+1}> d^{\prime \prime}_{2j+2}= d^{\prime \prime}_{2j+3}$, then $d^{\prime \prime}_{2j+3} \in S_{\text{even}}$ contributes a generator of $\bar{A}(\cO)$. If $d^{\prime \prime}_{2j+1}> d^{\prime \prime}_{2j+2} > d^{\prime \prime}_{2j+3}$, then $d^{\prime \prime}_{2j+2}$, $d^{\prime \prime}_{2j+3}$ belong to $S_{\text{odd}}$ and $d^{\prime \prime}_{2j+3}$  can also be chosen as a generator of $\bar{A}(\cO)$.
	    \item[--] In type $C_n$, if $d^{\prime \prime}_{2j+1} = d^{\prime \prime}_{2j+2}> d^{\prime \prime}_{2j+3}$, then $d^{\prime \prime}_{2j+2} \in S_{\text{even}}$ contributes a generator of $\bar{A}(\cO)$. If $d^{\prime \prime}_{2j+1}> d^{\prime \prime}_{2j+2} > d^{\prime \prime}_{2j+3}$, then $d^{\prime \prime}_{2j+1}$, $d^{\prime \prime}_{2j+2}$ belong to $S_{\text{odd}}$ and $d^{\prime \prime}_{2j+2}$ can also be chosen as a generator of $\bar{A}(\cO)$.
	\end{itemize}
\end{remark}

By \cite[Chapter13]{Lu84}, Lusztig's canonical quotient is preserved under the Springer dual map. We provide here a direct proof of this fact based on Proposition  \ref{p.CombCanonicalQuotient}.

\begin{lemma} \label{l.EqualityOfCanonicalQuotient}
	Let $\cO$ be a special orbit in $\mathfrak{sp}_{2n}$ and let ${}^S\cO$ be its Springer dual orbit in $\mathfrak{so}_{2n+1}$, then
	\begin{align*}
		\widebar{A}(\cO) \simeq  \widebar{A}({}^S\cO).
	\end{align*} 	
\end{lemma}
\begin{proof} 
Without loss of generality, we assume the partition of $\cO$ has the following form
\begin{align*}
	\mathbf{d}_C = [ \mathbf{d}_0, \mathbf{d}_1, \cdots, \mathbf{d}_{2r} ],
\end{align*} 
where
\begin{align*}
	& \mathbf{d}_j = [ d_{j,1}, \cdots, d_{j,N_j} ], \quad d_{j,1} \equiv \cdots \equiv d_{j,N_j} \equiv j, \quad \text{for} \quad 0 \leq j \leq 2r-1, \\
	& \mathbf{d}_{2r} = [ d_{2r,1} \equiv \cdots \equiv d_{2r, N_{2r}} \equiv 0]. 
\end{align*}
Since $\cO$ is special, then $N_j \equiv 0$, for $0 \leq j \leq 2r-1$. To obtain $\bar{A}(\cO)$, following the rule of Proposition \ref{p.CombCanonicalQuotient}, one first removes all $\mathbf{d}_{2j-1}$, for $1 \leq j \leq r$, and removes all even parts in $\mathbf{d}_{2j}$, for $0 \leq j \leq r$, of the form $d_{2j, 2l}=d_{2j, 2l+1}$.

By Proposition \ref{p.Springerdual}, the partition of the Springer dual orbit ${}^S\cO$ is given by $\mathbf{d}_B=[(\mathbf{d}_0)^+, \mathbf{d}_1, \cdots, \mathbf{d}_{2r}]_B$. By Lemma \ref{l.trick}, we have
\begin{align*}
	\mathbf{d}_B = [ \tilde{\mathbf d}_{0}, \mathbf{d}_1, \tilde{\mathbf d}_{2}, \cdots, \mathbf{d}_{2r-1}, \tilde{\mathbf d}_{2r}  ],
\end{align*}
where
\begin{align*}
	\tilde{\mathbf d}_{2j} = [d_{2j,1}+1, {\mathbf d}^{\prime}_{2j}, d_{2j, N_{2j}}-1 ], \quad  {\mathbf d}^{\prime}_{2j} = [d_{2j,2}, \cdots, d_{2j,N_{2j}-1} ]_B, \quad \text{for} \quad 0 \leq j \leq r.
\end{align*}
The $B$--collapse happens in ${\mathbf d}^{\prime}_{2j}$, where $ d_{2j, 2l} \neq d_{2j, 2l+1}$. To obtain $\bar{A}({}^S\cO)$, one first removes all even parts in $\tilde{\mathbf{d}}_{2j}$, for $0 \leq j \leq r$. Note that $\mathbf{d}_{2j-1}$ is coming from $\mathbf{d}_{C}$, then we have $d_{2j-1, 2l+1} = d_{2j-1, 2l+2}$, for $0 \leq l \leq N_{2j-1}/2-1 $. Thus, one removes all $\mathbf{d}_{2j-1}$, $1 \leq j \leq r$.

Let's compare the remaining parts in $\mathbf{d}_C$ and $\mathbf{d}_B$, i.e, comparing the number of remaining parts in $\mathbf{d}_{2j}$ and $\tilde{\mathbf d}_{2j}$, $0 \leq j \leq r-1$. Following the rule, the remaining parts in $\mathbf{d}_{2j}$ are pairs of the form $\{d_{2j,2l} \neq d_{2j,2l+1}  \}$ and $\{d_{2j,1}, d_{2j, N_{2j}} \}$. Similarly, the remaining parts in $\tilde{\mathbf d}_{2j}$ consist of pairs of the form $\{ d_{2j,2l}-1,  d_{2j,2l+1}+1 \}$, for $d_{2j,2l} \neq d_{2j,2l+1}$, and $\{d_{2j,1}-1, d_{2j, N_{2j}}+1 \}$. Thus, we have the following decompositions
\begin{align*}
	\mathbf{d}_C &=[d^{\prime \prime}_{1}, \cdots, d^{\prime \prime}_{2q+1}] \cup [ \mathbf{d}_1, \mathbf{d}_3, \cdots, \mathbf{d}_{2r-1}  ] \cup [\beta_1, \beta_1, \cdots, \beta_y, \beta_y], \\
	\mathbf{d}_B &=[\tilde{d^{\prime \prime}}_{1}, \cdots, \tilde{d^{\prime \prime}}_{2q+1}] \cup [\tilde{\alpha}_1, \tilde{\alpha}_1, \cdots, \tilde{\alpha}_x, \tilde{\alpha}_x] \cup [ \mathbf{d}_1, \mathbf{d}_3, \cdots, \mathbf{d}_{2r-1}  ].
\end{align*}
Finally, by Remark \ref{r.CombCanonicalQuotient}, there is a natural isomorphism between the generators of $\widebar{A}(\cO)$ and $\widebar{A}({}^S\cO)$. Then we conclude.
\end{proof}

\begin{proposition}\label{canonical-quotient-rigid-special}
 If $\0_{rs} \subset \g$ is a rigid special orbit, so is its Springer dual ${}^S\0_{rs}$.  Moreover, we have
\begin{align*}
	\pi_1(\0_{rs})=A(\0_{rs}) = \bar{A}(\0_{rs}) \simeq \bar{A}({}^S\0_{rs}) =A({}^S\0_{rs}) = \pi_1({}^S\0_{rs}).
\end{align*}
\end{proposition}
\begin{proof}
The first claim is Lemma \ref{rigid-special-under-Springer-dual}. For the second, consider the rigid special orbit $\cO_{rs}$ in $C_n$ with partition
	\begin{align*}
		\mathbf{d}_C=[k^{2n_k}, (k-1)^{2n_{k-1}},\cdots, 2^{2n_2}, 1^{2n_1}], \ \text{with} \ n_{\text{even}} \geq 2 \ \text{and}\ \ n_{\text{odd}} \geq 1.
	\end{align*}
	\begin{itemize}
		\item If $k$ is even, from the above rule, we have
		\begin{align*}
			\mathbf{d}_C= [k^2,(k-2)^2,\cdots, 2^2] & \cup [(k-1)^{2n_{k-1}}, (k-3)^{2n_{k-3}}, \cdots, 1^{2n_1}] \\                        & \cup [k^{2n_k-2}, (k-2)^{2n_{k-2}-2}, \cdots, 2^{2n_2-2}].
		\end{align*}
		\item If $k$ is odd, from the above rule, we have
		\begin{align*}
			\mathbf{d}_C= [(k-1)^2,(k-3)^2,\cdots, 2^2] & \cup [k^{2n_{k}}, (k-2)^{2n_{k-2}}, \cdots, 1^{2n_1}] \\                        & \cup [(k-1)^{2n_{k-1}-2}, \cdots, 2^{2n_2-2}].
		\end{align*}
	\end{itemize}
It follows that $\widebar{A}(\cO_{rs})\simeq \mathbb{Z}_2^{[\frac{k}{2}]}$. Note that $[\frac{k}{2}]$ is the  number of distinct even parts in $\mathbf{d}_C$, hence $\widebar{A}(\cO_{rs}) = A(\cO_{rs}) = \pi_1(\cO_{rs})$. A similar argument applies to rigid special orbits in $B_n$. We conclude the proof by Lemma \ref{l.EqualityOfCanonicalQuotient}.
\end{proof}

\section{Mirror symmetry for parabolic covers of Richardson orbits} \label{MSPCRO}

\subsection{Springer duality and Lusztig's canonical quotient} \label{DSM}

We start with the following result, which says that Springer duality and Langlands duality are compatible under Springer maps.

\begin{proposition} \label{p.Springer-Richardson}
Let $\cO$ be a Richardson orbit of type $C_n$ and ${}^S\cO$ its Springer dual orbit of type $B_n$. Then ${}^S\cO$ is Richardson, and for any polarization $P$ of $\cO$, its Langlands dual ${}^LP$ is a polarization of ${}^S\cO$.
In other words, we have the following relation between Springer duality and Langlands duality:
\begin{align} \label{Springer-dual}
\xymatrix{T^*( G/P )  \ar[d]_{\nu_P} & \stackrel{Langlands \  dual}{\rightsquigarrow } & T^* ( {}^LG / {{}^L P} )  \ar[d]^{\nu_{{}^LP}}  \\   \overline{ \mathcal{O}} &\stackrel{Springer  \ dual}{\rightsquigarrow }  & \overline{{}^S{ \mathcal{O}}} }.
\end{align}
\end{proposition}
\begin{proof}
Let $\mathbf{d}_C $ (resp. $\mathbf{d}_B$) be the partition of the Richardson orbit $\cO$ (resp. ${}^S\cO$). By Proposition \ref{p.Springerdual}, we know $\mathbf{d}_B=((\mathbf{d}_C)^+)_B$. Denote by
\begin{align}
\mathbf{d}=[ d_{1} \equiv \ldots \equiv d_{2k} \equiv 1 ,  d_{2k+1} \equiv \ldots \equiv d_{2r+1}\equiv 0 ]  , 	\nonumber
\end{align}
the induced partition from $P$ as in (\ref{induction}). Then $\mathbf{d}_C$ is the $C$-collapse of $\mathbf{d}$. Note that the induced partition from ${}^LP$ is given by
\begin{align}
{}^L\mathbf{d}=[ d_{1} \equiv \ldots \equiv d_{2k} \equiv d_{2k+1}+1 \equiv 1 ,  d_{2k+2} \equiv \ldots \equiv d_{2r+1} \equiv 0 ]. 	\nonumber
\end{align}
By Lemma \ref{l.trick}, we know that $((\mathbf{d}_C)^+)_B =({}^L\mathbf{d})_B $, so $\mathbf{d}_B= ({}^L\mathbf{d})_B$. Then we arrive at a conclusion.
\end{proof}

Given a polarization $P$ of the  Richardson orbit $\cO$,  we denote by $\mathcal{A}_P = P_x/G_x^\circ$ its stabilizer in $\mathcal{A}(\cO)$. Consider the following exact sequence defining Lusztig's canonical quotient
\begin{align*}
	1 \longrightarrow K(\0) \longrightarrow \mathcal{A}(\cO) \longrightarrow \bar{A}(\cO) \longrightarrow 1.
\end{align*}

\begin{proposition} \label{p.A_P}
Let $\0$ be a Richardson orbit in classical Lie algebra except type $A$, and $P$ be a polarization. Then the kernel $K(\0) $ is a subgroup of $\mathcal{A}_P$. 
\end{proposition}

Before giving the proof, we need to introduce some notations following \cite{So, He78}. Let $\mathbf{d}_{\varepsilon}=[d_1, d_2, \dots, d_{N}]$ be a special partition corresponding to the special orbit $\cO$ of type $B$, $C$ or $D$. Then we have $S_{\text{odd}}$ and $S_{\text{even}}$ defined in Section \ref{LCQ}. List the elements of $S_{\text{odd}}$ in decreasing order $i_l \geq i_{l-1} \geq \cdots \geq i_1$. Note that $l$ is automatically odd in type $B$ and automatically even in type $D$. Assume that $l$ is even in type $C$ by setting $i_1=0$ if necessary.

Given $x \in \mathcal{A}(\cO)$, write $x= x_{m_k} + x_{m_{k-1}} + \dots + x_{m_1}$ where $m_k \geq m_{k-1} \geq \dots \geq m_1$. From $x$ we can construct two sets $T_1$ and $T_2$ with $T_1 \subset S_{\text{odd}}$ and $T_2 \subset S_{\text{even}}$ such that $T_1 \cup T_2 = \{m_k, m_{k-1}, \dots, m_1 \}$. For $i_j \in S_{\text{odd}}$, let $\delta_j = 1$ if $i_j \in T_1$ and $0$ if $i_j \notin T_1$. We assume that $\delta_l=0$ in type $C$ and $D$. Finally, we define subsets $T_2^{(m)}$ of $T_2$ as follows: let $T_2^{(m)}$ consist of those $j \in T_2$ such that $i_{m+1} \geq j \geq i_{m}$ (with convention $i_0 =0$) and let $t_m=|T_2^{(m)}|$. From \cite[Proposition 4, Theorem 6]{So}, we know that the group $K(\0)$ consists of the following elements:
\begin{align*}
	 K(\0)=\{ x \in \mathcal{A}(\cO) \mid t_m=0,\ \text{when} \ m \equiv 0;\ \delta_{m+1}+t_m+\delta_{m} \equiv 0, \ \text{when} \ m \equiv 1 \}.
\end{align*}

\begin{example}
	Consider the nilpotent orbit $\cO$ of type $C$ with partition $\mathbf{d}_C=[9^2,8,6^2,4,3^2,2^2]$. Then $\bar{A}(\0) \simeq \mathbb{Z}_2^{\oplus 2}$ and
	\begin{align*}
		B(\mathbf{d}_{C})= \{8, 6, 4, 2 \}, \quad S_{\text{odd}}=\{i_2 \geq i_1 \} := \{ 8, 4 \},\ \text{and} \ S_{\text{even}}= \{ 6, 2 \}.
	\end{align*}
Let $x=x_8+x_6+x_2 \in \mathcal{A}(\cO)$. We have
	\begin{align*}
		T_1= \{ 8 \}, \ T_2= \{ 6, 2 \}.
	\end{align*}
	Then
	\begin{align*}
		\delta_j = \begin{cases}
			1, \quad \text{if} \quad j=2, \\
			0, \quad \text{if} \quad j=1,
		\end{cases}
	\end{align*}
	and
	\begin{align*}
		T_2^{(0)} = \{ 2 \}, \quad T_2^{(1)} = \{ 6 \}.
	\end{align*}
	As $t_0 \neq 0$, the element  $x$ does not belong to $K(\cO)$.
\end{example}

For a Richardson orbit $\0$ with a given  polarization $P$, let
\begin{align*}
 \mathbf{d}=[d^{\prime}_1, \cdots, d^{\prime}_{2r+1}]
\end{align*}
be the induced partition as (\ref{induction}) of Section 2.3. Then, $\mathbf{d}_\varepsilon $, the $\varepsilon$--collapse of $\mathbf{d}$, is the partition of $\cO$. Set
\begin{align*}
	I_{\chi}(\mathbf{d})=\{ j \in \mathbb{N} \mid j \equiv  d^{\prime}_j \equiv \chi,\  d^{\prime}_j \geq d^{\prime}_{j+1}+2 \}.
\end{align*}
Note that the set $I_{\chi}(\mathbf{d})$ contains the parts where $\varepsilon$--collapses happen. Then $\mathcal{A}_P$ consists of $x=x_{i_k} + \cdots + x_{i_1} \in \mathcal{A}(\cO)$ where $x_{d_j}$ and $x_{d_{j+1}}$ appear simultaneously whenever $j \in I_{\chi}(\mathbf{d})$, for details, see \cite[Theorem 7]{He78}.\\

Now we can prove Proposition \ref{p.A_P}.\\

\begin{proof}
We give the proof for type $B$, which is similar for type $C$ and $D$. In this case
\begin{align*}
    \mathbf{d}=[ d^{\prime}_1 \equiv \ldots \equiv d^{\prime}_{2l-1} \equiv 1,    d^{\prime}_{2l} \equiv \ldots \equiv d^{\prime}_{2r+1}\equiv 0 ],
\end{align*}
the $B$--collapse happens at parts $d^{\prime}_{2a}$  such that for $2l \leq 2a \leq 2r $ and  $d_{2a}^{\prime} \neq d_{2a+1}^{\prime}$, which makes the following changes:
\begin{align*}
	d_{2a}^{\prime} \mapsto d_{2a}:=d_{2a}^{\prime}-1, \quad d_{2a+1}^{\prime} \mapsto d_{2a+1}:=d_{2a+1}^{\prime}+1.
\end{align*}
Note that in type $B$, $S_{\text{odd}}$ consists of some odd parts in $\mathbf{d}_B$.

If $d_{2a}=d_{2a+1}$, note $ \chi=0 \notequiv d_{2a+1} > d_{2a+2}^\prime $. Then $d_{2a+1} \in B(\mathbf{d}_B)$, but $ d_{2a} \notin B(\mathbf{d}_B)$. So if $x=x_{i_k} + \cdots + x_{i_1} \in \mathcal{A}_P$, $x_{d_{2a+1}}$ does not appear in $x$. Moreover, since $[d^{\prime}_1, \ldots, d^{\prime}_{2l-1}]$ consists of an odd number of odd parts and together with the rule of the collapse of type $B$, there exists an even number $m$, such that $i_m, i_{m+1} \in S_{\text{odd}}$ and $ i_{m+1} > d_{2a}=d_{2a+1} > i_{m}$. Then by the definition of $K(\cO)$, $x_{d_{2a+1}}$ does not appear in $x \in  K(\cO)$.

If $d_{2a}>d_{2a+1}$, we claim that if $x \in K(\0)$, then $x_{d_{2a}}, x_{d_{2a+1}}$ appear in $x$ simultaneously. Suppose $d_{2a}=i_{m+1} \in S_{\text{odd}} $, then $d_{2a+1}= i_{m} \in S_{\text{odd}}$. Note that $[d^{\prime}_1, \ldots, d^{\prime}_{2l-1}]$ consists of an odd number of odd parts, then $m \equiv 1$. If only one of $x_{d_{2a}}$ and $x_{d_{2a+1}}$ appears in $x$. Then $\delta_{m+1}+t_{m}+\delta_{m}=1$, which contradicts to $x \in K(\0)$. Thus $K(\0) < \mathcal{A}_P$.
\end{proof}

 Let $P$ be a polarization of a Richardson orbit $\cO$ of type $C$ and denote by
 \begin{align*}
    \mathbf{d}=[ d^{\prime}_1 \equiv \ldots \equiv d^{\prime}_{2k}\equiv 1 ,  \mid  d^{\prime}_{2k+1} \equiv \ldots \equiv d^{\prime}_{2r+1} \equiv 0 ],
\end{align*}
the partition induced from $P$. Then the partition of $\cO$ is given by $\mathbf{d}_C$, the $C$--collapse of $\mathbf{d}$. Let ${}^L\mathbf{d}$ be the partition induced from  ${}^LP$, i.e.,
\begin{align*}
    {}^L\mathbf{d}=[ d^{\prime}_1 \equiv \ldots \equiv d^{\prime}_{2k} \equiv d^{\prime}_{2k+1}+1 \equiv 1,  \mid   d^{\prime}_{2k+2} \equiv \ldots \equiv d^{\prime}_{2r+1}\equiv 0].
\end{align*}
Let
\begin{align*}
	\mathbf{d}^{\prime} :=\{ d^{\prime}_{j} \in \mathbf{d}  \mid j \in I_{\chi=1}(\mathbf{d}) \ \text{or} \ j-1 \in I_{\chi=1}(\mathbf{d})   \}.
\end{align*}
Similarly, we have ${}^L\mathbf{d}^{\prime}$. 

Recall in Proposition \ref{p.CombCanonicalQuotient}, by abuse of notation; we have
\begin{align*}
	\mathbf{d}_{\varepsilon} = [d^{\prime \prime}_{1, \varepsilon}, \cdots, d^{\prime \prime}_{2q+1, \varepsilon}] \cup [\alpha_{1,\varepsilon}, \alpha_{1, \varepsilon}, \cdots, \alpha_{x, \varepsilon}, \alpha_{x, \varepsilon}] \cup [\beta_{1, \varepsilon}, \beta_{1, \varepsilon}, \cdots, \beta_{y, \varepsilon}, \beta_{y, \varepsilon}],
\end{align*}
here
\begin{align*}
	\mathbf{d}_\varepsilon = \begin{cases}
		 ({}^L\mathbf{d})_B, \quad &\text{for} \quad \varepsilon=0, \\
		\mathbf{d}_C, \quad &\text{for} \quad \varepsilon=1.
	\end{cases}
\end{align*}
Let $\mathbf{d}_{q,\varepsilon} := [d^{\prime \prime}_{1, \varepsilon}, \cdots, d^{\prime \prime}_{2q+1, \varepsilon}]$. Using the above notations, we arrive at the following lemma.

\begin{lemma} \label{l.quotient-group}
Let $\mathbf{d}_{q,\chi} := [\mathbf{d}^{\prime} \cup \{d_{2k+1}^\prime+1-\chi  \}  \cup {}^L\mathbf{d}^{\prime}]$, for $\chi=0$ (resp. $\chi=1$) when $\varepsilon=1$ (resp. $\varepsilon=-1$). Then
	\begin{align*}
		\mathbf{d}_{q, \varepsilon}= (\mathbf{d}_{q,\chi})_\varepsilon.
	\end{align*}
	Here $( - )_\varepsilon$ is the $B$(resp. $C$)--collapse when $\varepsilon=1$ (resp. $-1$).
\end{lemma}
\begin{proof}
We prove for type $C$, i.e., $\chi=1$, as the situation for type $B$ is similar. From the rule of $C$--collapse, we have $\mathbf{d}_C= [d^{\prime}_1 \equiv \ldots \equiv d^{\prime}_{2k} \equiv 1]_C \cup [d^{\prime}_{2k+1} \equiv \ldots \equiv d^{\prime}_{2r+1} \equiv 0]_B$. Denote by
\begin{align*}
	[\tilde{d^{\prime}}_{1}, \ldots, \tilde{d^{\prime}}_{2k}]:=[d^{\prime}_{1}, \ldots, d^{\prime}_{2k}]_C.
\end{align*}
Following the rule in Proposition \ref{p.CombCanonicalQuotient}, to obtain $[d^{\prime \prime}_{1}, \cdots, d^{\prime \prime}_{2q+1}]$ from $\mathbf{d}_C$, we first remove all the odd parts in $[\tilde{d^{\prime}}_{1}, \ldots, \tilde{d^{\prime}}_{2k}]$. Note that the $C$--collapse of $[\tilde{d^{\prime}}_{1}, \ldots, \tilde{d^{\prime}}_{2k}]$ happens at $d_{2j+1}^\prime \neq d_{2j+2}$. Then the remaining even parts in $[\tilde{d^{\prime}}_{1}, \ldots, \tilde{d^{\prime}}_{2k}]$ are all left in $[d^{\prime \prime}_{1}, \cdots, d^{\prime \prime}_{2q+1}]$, and they are the $C$--collapse of $\mathbf{d}^{\prime}$. 

Then we only need to see which even parts left in $[d^{\prime}_{2k+1}, \ldots, d^{\prime}_{2r+1}] $. Following the rule in Proposition \ref{p.CombCanonicalQuotient}, we remove even parts of the form ${d^{\prime}}_{2j}={d^{\prime}}_{2j+1}$. Note that the $B$--collapse of $[d^{\prime}_{2k+1}+1, \ldots, d^{\prime}_{2r+1}]$, which happens at $d^{\prime}_{2j} \neq d^{\prime}_{2j+1}$. Then the remaining even parts are $ \{ d_{2k+1}^\prime \} \cup {}^L\mathbf{d}^\prime$. We conclude.
\end{proof}

%\begin{remark}
%	Notice that in type $C_n$, $d_{2k+1}^{\prime}$ remains in $[d^{\prime \prime}_{1}, \cdots, d^{\prime \prime}_{2q+1}]$. In type $B_n$, $d_{2k+1}^{\prime}+1$ remains. Then $d_{2k+1}^{\prime}$ (or $d_{2k+1}^{\prime}+1$) divides $\mathbf{d}_{q,\chi}$ into two sets such that one of the two sets needs to be collapsed and another one does not.
%\end{remark}

Let $\cO$ be a Richardson orbit of type $B$ or $C$, and let  $P$ be a polarization of $\0$. Let $\bar{A}_P := \mathcal{A}(\cO)/ \mathcal{A}_P $, from Proposition \ref{p.A_P}, we have a surjective map $\bar{A}(\0) \to  \bar{A}_P$. Then we can write $\mathcal{A}(\cO) \cong \mathcal{A}_P \times \bar{A}_P $. We know that ${}^LP$ is a polarization of ${}^S \cO$ and $\mathcal{A}({}^S\cO)=\mathcal{A}_{{}^LP} \times \bar{A}_{{}^LP}$.

\begin{proposition} \label{p.LusztigCanonical}
	Under the isomorphism $\bar{A}(\cO)\simeq \bar{A}({}^S\cO)$, we have
	\begin{align*}
		\bar{A}(\cO) \cong \bar{A}_P \times \bar{A}_{{}^LP}.
	\end{align*}
\end{proposition}
\begin{proof}
	Using the isomorphism between $\bar{A}(\cO)$ with $\bar{A}({}^S\cO)$ and we can use elements in $\mathbf{d}_{q,\chi=1}$ to describe the generators. Suppose $\cO$ is of type $C$. From Lemma \ref{l.quotient-group}, we know $\mathbf{d}_{q,\chi=1}$ consists of the following even pairs:  
	\begin{itemize}
		\item[(i)] even pairs $\{d_{2j+1}^{\prime}-1, d_{2j+2}^{\prime}+1 \}$ in $[d^{\prime}_1 \equiv \ldots \equiv d^{\prime}_{2k}  \equiv 1]_C$ such that $d_{2j+1}^{\prime} \neq d_{2j+2}^{\prime}$
		\item[(ii)] even pairs $\{d_{2l}^{\prime}, d_{2l+1}^{\prime} \}$ in $[d^{\prime}_{2k+1} \equiv \ldots \equiv d^{\prime}_{2r+1} \equiv 0]$ such that $d_{2l}^{\prime} \neq d_{2l+1}^{\prime}$,
		\item[(iii)] $\{d_{2k+1}^{\prime} \}$.
	\end{itemize}
	From Remark \ref{r.CombCanonicalQuotient}, $d_{2j+2}^{\prime}+1$ in even pair $\{d_{2j+1}^{\prime}-1, d_{2j+2}^{\prime}+1 \}$, for $ j \leq k-1$, contributes to a generator of $\bar{A}(\cO)$, and $d_{2l}^{\prime}$ in even pair $\{d_{2l}^{\prime}, d_{2l+1}^{\prime} \}$, for $ l > k $, contributes to a generator of $\bar{A}(\cO)$. Note that $d_{2k+1}^{\prime}$ is not chosen as a generator. 
	
	The description of $\mathcal{A}_P$ shows that $\bar{A}_P$ consists of $\{ d_{2j+2}^{\prime}+1 \}$, for $j \leq k-1$, such that $ d_{2j+1}^{\prime} \neq d_{2j+2}^{\prime} $. Similarly, $\bar{A}_{{}^LP}$ consists of $\{ d_{2l}^{\prime} \}$, for $l > k $, such that $d_{2l}^{\prime} \neq d_{2l+1}^{\prime}$. Then we arrive at a conclusion.
\end{proof}

Using notations in Proposition \ref{p.Springer-Richardson}, together with the above proposition, we have
\begin{corollary} \label{c.DegreeDuality}
	\begin{align*}
		\deg \nu_P \cdot \deg \nu_{{}^LP} = |\bar{A}(\cO)| = |\bar{A}({}^S\cO)|.
	\end{align*}
\end{corollary}
\begin{proof}
	Let $x \in \cO$, then the degree of map $\nu_P$ in (\ref{Springer-dual}) is $|G_x/P_x|$. By definition, $\mathcal{A}(\cO)= G_x/G_x^\circ$ and $\mathcal{A}_P = P_x/G_x^\circ$, which gives  $\deg \nu_P =|\mathcal{A}(\0)/\mathcal{A}_P| =  |\bar{A}_P|$. This concludes the proof by Proposition \ref{p.LusztigCanonical}.
\end{proof}

By Corollary \ref{c.DegreeDuality}, if a Richardson orbit $\0$ of type $B$ or $C$ satisfies $\bar{A}(\0)=1$, then $\overline{\0}$ admits a crepant resolution. In the following, we investigate the relationship between the existence of a crepant resolution of $\overline{\0}$ and Lusztig's canonical quotient for special orbits.

\begin{proposition}
	Let $\cO_{\mathbf{d}}$ be a special nilpotent orbit in a simple classical Lie algebra. If $\bar{A}(\cO)=1$, then ${\cO}_{\mathbf{d}}$ is Richardson and $\overline{\cO}_{\mathbf{d}}$ has a crepant resolution.\end{proposition}
\begin{proof}
The case of type $A$ is trivial, as $\bar{A}(\0)=1$ for all $\0$, and each nilpotent orbit closure admits a crepant resolution. Now we do a case-by-case check for other classical types:

{\textbf{Type $B_n$:}} From Proposition \ref{p.CombCanonicalQuotient}, $\bar{A}(\cO_{\mathbf{d}})=1$ implies $q=0$. From the rule of separating even and odd parts in the partition $\mathbf{d}$, there is only one odd part remained, i.e., partition $\mathbf{d}$ has the form
		\begin{align*}
			\mathbf{d}_B=[d_{1}^{2i_1} \equiv \cdots \equiv d_{j}^{2i_j} \equiv d_{j+1}\equiv 1,  d_{j+2}^{2i_{j+2}} \equiv \cdots \equiv d_{s}^{2i_s}\equiv 0].
		\end{align*}
	It follows that $\0$ is a Richardson orbit with $\bar{A}(\0)=1$, hence  $\overline{\cO}_{\mathbf{d}}$  has a crepant resolution.

		{\textbf{Type $C_n$:}} Similarly, in this case, partition $\mathbf{d}$ has the form
		\begin{align*}
			\mathbf{d}_C=[d_{1}^{2i_1} \equiv \cdots \equiv d_{j}^{2i_j}\equiv 1,  d_{j+1} \equiv d_{j+2}^{2i_{j+2}} \equiv \cdots \equiv d_{s}^{2i_s}\equiv 0],
		\end{align*}
		or
		\begin{align*}
			\mathbf{d}_C=[ d_1 \equiv d_2^{2i_2} \equiv \cdots \equiv d_s^{2i_s}\equiv 0].
		\end{align*}
		Then ${\cO}_{\mathbf{d}}$ is Richardson, and $\overline{\cO}_{\mathbf{d}}$ has a crepant resolution.
		
		{\textbf{Type $D_n$:}} Special partitions in $\mathcal{P}_{1}(2n)$ are similar to special partitions in $\mathcal{P}_{1}(2n+1)$,  except that we require there is an even number of odd parts greater than the largest even part. For the combinatorial description of Lusztig's canonical quotient, see  \cite[Proposition 2.2]{Won17}. It follows that if $\bar{A}(\0)=1$, then the  partition $\mathbf{d}$ has the form
		\begin{align*}
			\mathbf{d}_D=[ d_1^{2i_1} \equiv d_2^{2i_2} \equiv \cdots \equiv d_s^{2i_s}\equiv 0].
		\end{align*}
		Then ${\cO}_{\mathbf{d}}$ is an even orbit, so $\overline{\cO}_{\mathbf{d}}$ has  a crepant resolution.	
\end{proof}

\begin{proposition}
	Let $\cO$ be a Richardson orbit in a simple Lie algebra. If $\bar{A}(\cO)=1$, then $\overline{\cO}$ has a crepant resolution.
\end{proposition}
\begin{proof}
	For exceptional types, by \cite[Corollary 5.11]{Fu4}, if a Richardson orbit closure $\overline{\cO}$ has no crepant resolution, then $\cO$ is one of the following orbits:
		\begin{align*}
			A_4+A_1, \quad D_5(a_1)  \ \text{in} \ \ E_7, \quad  \text{or} \quad  E_6(a_1) + A_1, \quad E_7(a_3)  \ \text{in}\ E_8.
		\end{align*}
	However, all of the above cases have $\bar{A}(\cO)=\mathbb{Z}_2$.
\end{proof}
\begin{remark}
The claim fails for special nilpotent orbits in exceptional types. For example the simply-connected orbit $A_1 + \tilde{A}_1$ in $F_4$ is special but not Richardson, so it has no crepant resolution by \cite{Fu1}.
\end{remark}

\subsection{Mirror symmetry for parabolic covers}

The Hodge-Deligne polynomial of a smooth variety $Z$ is defined as
\begin{align*}
	{\rm E}(Z; u, v)=\sum_{p, q} \sum_{k \geq 0}(-1)^{k} h^{p, q}\left(H_{c}^{k}(Z, \mathbb{C})\right) u^{p} v^{q},
\end{align*}
where $h^{p, q}\left(H_{c}^{k}(Z, \mathbb{C})\right)$ is the dimension of $(p, q)$--th Hodge-Deligne component in the $k$--th cohomology group with compact supports. Moreover, if $f: Z \rightarrow Y$ is a Zariski locally trivial fibration with fiber $F$ over a closed point, then from the exact sequence of mixed Hodge structures, we have ${\rm E}(Z)={\rm E}(Y) \times {\rm E}(F)$.

The following gives formulae for Hodge-Deligne polynomials of classical Grassmannians, which should be well-known (e.g., \cite[Proposition A.4]{BL19} for the case of Grassmannians).

 \begin{proposition} \label{p.EPolyGr}
 The E-polynomials of Grassmannians, isotropic Grassmannians, and orthogonal Grassmannians are given by the following (where $q=uv$):
 \begin{align*}
	& {\rm E}({\rm Gr}(k,n)) = \prod_{j=1}^k \frac{(q^{n-j+1}-1)}{(q^{k-j+1}-1)} =\prod_{j=1}^k \frac{(q^{n-k+j}-1)}{(q^j-1)}, \\
	& {\rm E}({\rm IG}(k,2n)) = {\rm E}({\rm OG}(k,2n+1)) = \prod_{j=1}^k \frac{(q^{2n-2k+2j}-1)}{(q^j-1)}, \\
	& {\rm E}({\rm OG}(k,2n) = \prod_{j=1}^k \frac{(q^{n-k-1+j}+1)(q^{n-k+j}-1)}{(q^j-1)}.
\end{align*}
\end{proposition}
 \begin{proof}
 The proof is standard, and we just do it for ${\rm OG}(k, 2n)$.  Recall that the $E$--polynomial of a smooth hyperquadric $\mathbb{Q}^{2k}$ is given by:
 $$
 {\rm E}(\mathbb{Q}^{2k}) =q^k + \sum_{j=0}^{2k} q^j  =  \cfrac{(q^k+1)(q^{k+1}-1)}{q-1}.
 $$
 The full flag variety  ${\rm SO}_{2n}/B$ has a fibration over $\mathbb{Q}^{2n-2}$ with fiber being the full flag variety ${\rm SO}_{2n-2}/B$. It follows that
 $$
 {\rm E}({\rm SO}_{2n}/B) = {\rm E}({\rm SO}_{2n-2}/B) \times {\rm E}(\mathbb{Q}^{2n-2}) = \cfrac{1}{(q-1)^{n-1}} \prod_{j=1}^{n-1}(q^j+1)(q^{j+1}-1).
 $$
 In a similar way, we obtain the E-polynomial for the full flag variety of ${\rm SL}$:
 $${\rm E} ({\rm SL}_k/B) = \cfrac{1}{(q-1)^{k-1}} \prod_{j=1}^{k-1}(q^{j+1}-1).$$

 Now we consider the fibration ${\rm SO}_{2n}/B \to {\rm OG}(k, 2n)$, whose fiber is the product ${\rm SL}_k/B \times  {\rm SO}_{2n-2k}/B$. Then a direct computation gives the formula of ${\rm E}({\rm OG}(k, 2n))$.
 \end{proof}

Let $W$ be a normal variety with only canonical singularities. Consider a log resolution of singularities $\rho: Z \rightarrow W$, namely $\rho$ is a resolution such that the exceptional locus of $\rho$ is a divisor whose irreducible components $D_{1}, \cdots, D_{k}$ are smooth with only normal crossings. As $W$ has only canonical singularities, we have
$$
K_{Z}=\rho^* K_W + \sum_{i} a_{i} D_{i}, \ {\text with}\  a_{i} \geq 0.
$$
For any subset $J \subset I:=\{1, \cdots, k\}$, one defines $D_{J}=\bigcap_{j \in J} D_{j}, D_{\emptyset}=Z$ and $D_{J}^{0}=$ $D_{J}-\bigcup_{i \in I-J} D_{i}$, which are all smooth. Then the \emph{stringy E-function} of $W$ is defined by:
\begin{align*}
{\rm E}_{\rm st}(W ; u, v)=\sum_{J \subset I} {\rm E}\left(D_{J}^{0} ; u, v\right) \prod_{j \in J} \frac{u v-1}{(u v)^{a_{j}+1}-1} .	
\end{align*}

We recall the following basic properties of stringy E-functions \cite{Ba98}.
\begin{proposition} \label{p.Batyrev}
\begin{itemize}
	\item[(1)] The stringy E-function is independent of the choice of the log resolution.
	\item[(2)] If $F$ is a smooth variety, then ${\rm E}_{\rm st}(W \times F; u,v) = {\rm E}_{\rm st}(W; u,v) \ {\rm E}(F; u,v)$.
	\item[(3)] If $\rho: Z \to W$ is a crepant resolution,  then  ${\rm E}_{\rm st}(W; u, v)={\rm E}(Z; u, v)$.
\end{itemize}
\end{proposition}

By \cite{Be00}, the normalization $\widetilde{\0}$ of $\overline{\0}$ has only canonical singularities. It is then interesting to compute its stringy E-function, which has never been investigated so far. By \cite{Fu1}, any crepant resolution of $\widetilde{\0}$ is isomorphic to a Springer resolution; in particular, $\0$ must be a Richardson orbit. In this case, it follows that the stringy E-function of $\widetilde{\0}$ is actually a polynomial.

Conversely, consider a Richardson orbit $\cO$ with a polarization $P$.  The corresponding  Springer map $\nu: T^*(G/P) \to \widebar{\cO}$ is generically finite.  By taking the Stein factorization, we have
\begin{align*}
	\nu: T^*(G/P)  \xrightarrow{\pi_P} X_P \xrightarrow{\mu_P} \overline{\0},
\end{align*}
where $\pi_P$ is birational, and $\mu_P$ is a finite map. We call $X_P$ the parabolic cover of $\0$ associated with $P$, which is normal with only canonical singularities.

The Springer dual orbit ${}^L\0$ is also a Richardson orbit which admits ${}^LP$ as a polarization. This gives again a parabolic cover $\mu_{{}^LP}: X_{{}^LP} \to \overline{{}^S\0} $. It turns out that $(X_P, X_{{}^LP})$ is a mirror pair in the following sense:
\begin{proposition}[Topological mirror symmetry] \label{top.ms}
For any polarization $P$ of a Richardson orbit $\0$, the two Springer dual parabolic covers $X_P$ and $X_{{}^LP}$ share the same stringy E-polynomial.
\end{proposition}
\begin{proof}
As $T^*(G/P)$ has trivial canonical bundle,  the birational map $\pi_P$ is a crepant resolution. By Proposition \ref{p.Batyrev}, we have
\begin{align*}
	{\rm E}_{\rm st} (X_P; u, v) = {\rm E}(T^*(G/P); u, v) ={\rm E}(G/P; u,v) (uv)^n,
\end{align*}
where $n=\dim G/P$. In a similar way, we obtain
\begin{align*}
	{\rm E}_{\rm st} (X_{{}^LP}; u, v) = {\rm E}(T^*({}^LG/{}^LP); u, v) = {\rm E}({}^LG/{}^LP; u,v) (uv)^n.
\end{align*}
Let $\p=\fl \oplus \n $ be the Lie algebra of $P$ with Levi type $(p_1, \dots, p_l ; 2q)$. Consider a parabolic subgroup $P^\prime$ with Levi type $(p ;q)$ where $p=p_1+ \dots + p_l$. With the notation $G={\rm Sp}_{sn}$ (resp. ${}^LG = {\rm SO}_{2n+1}$), we have $G/P^{\prime}= {\rm IG}(p,2n)$ (resp. ${}^LG/{}^LP^{\prime} = {\rm OG}(p,2n+1)$). There is a $G$ (resp. ${}^LG$)--equivariant map $f: G/P \rightarrow G/P^{\prime}$ (resp. ${}^Lf: {}^LG/{}^LP \rightarrow {}^LG/{}^LP^{\prime}$) with fiber $P^{\prime}/P$ (resp. ${}^LP^{\prime}/{}^LP$) being the partial flag variety ${\rm Fl}(p_1, \dots, p_l)$ of type $A$. Then
\begin{align*}
	& {\rm E}(G/P) = {\rm E}({\rm IG}(p,2n)) \times {\rm E}({\rm Fl}(p_1, \dots, p_l)), \\
	& {\rm E}({}^LG/{}^LP) = {\rm E}({\rm OG}(p,2n+1)) \times {\rm E}({\rm Fl}(p_1, \dots, p_l)).
\end{align*}
We conclude the proof by Proposition \ref{p.EPolyGr}.
\end{proof}

This completes the proof of Theorem \ref{t.Richardson}: as (1) follows from Proposition \ref{p.Springer-Richardson} and Lemma \ref{l.EqualityOfCanonicalQuotient}, (2) follows from Corollary \ref{c.DegreeDuality} and (3) follows from Proposition \ref{top.ms}.

Let us end this subsection with the following example, which completes that from Section \ref{s.Ric}.
\begin{example} \label{e.kbig}
Consider the isotropical Grassmannian ${\rm IG}(k, 2n)$  and the associated Springer map $\nu_k: T^*{\rm IG}(k, 2n) \to \overline{\0}_k$.  The dual part is the orthogonal Grassmannian ${\rm OG}(k,2n+1)$ with the dual Springer map  ${}^L \nu_k: T^*{\rm OG}(k, 2n+1) \to \overline{{}^S\0}_k$.
Assume $n-1\geq k \geq \frac{2n+1}{3}$ in the following.

The map  $\nu_k$ is always birational and $\0_k$ is given by the partition $[3^{2n-2k}, 2^{3k-2n}]$.  In this case, the closure $\overline{\0}_k$ is normal by \cite{KP82}. It follows that $\pi_1(\0_k) = A(\0_k) =\mathbb{Z}_2$ while Lusztig's canonical quotient $\bar{A}(\0_k)$ is trivial for $k$ odd and $\mathbb{Z}_2$ for $k$ even.

When $k$ is odd, then ${}^S\0_k$ is given by $[3^{2n+1-2k}, 2^{3k-2n-1}]$ and the Springer map ${}^L\nu_k$ is birational.  In this case,  $A({}^S\0_k) =\bar{A}({}^S\0_k) $ is trivial and the closure $\overline{{}^S\0_k}$ is non-normal if $k > \frac{2n+1}{3}$ and normal if $k=\frac{2n+1}{3}$ by \cite{KP82}.
 When $k$ is even, the orbit ${}^S\0_k$ has partition $[3^{2n+1-2k}, 2^{3k-2n-2}, 1^2]$ and  the Springer map ${}^L\nu_k$ is of degree 2.   In this case, we have $\pi_1({}^S\0_k) = A({}^S\0_k) =\bar{A}({}^S\0_k)  =\mathbb{Z}_2$ and the closure $\overline{{}^S\0_k}$ is normal by \cite{KP82}.  It follows that if $k$ is odd, then $(\overline{\0}_k, \widetilde{{}^S\0}_k)$ is a mirror pair, while for $k$ even, the pair $(\overline{\0}_k, {}^LX_k)$ is a mirror pair, where ${}^LX_k \to \overline{{}^S\0}_k$ is the double cover.

Similar to Question \ref{question},  when $k$ is even, let $X_k \to \overline{\0}_k$ be the double cover.  The pair $(X_k, \widetilde{{}^S\0}_k)$ is out the scope of Proposition \ref{top.ms}, hence it is unclear if this is a mirror pair.
\end{example}

\subsection{Other examples of mirror pairs beyond parabolic covers}

 For a nilpotent orbit closure $\overline{\0}$, a {\em symplectic cover} is a generically finite morphism $Z \to \overline{\0}$ such that $Z$ is smooth and symplectic. The existence of a symplectic cover is equivalent to the existence of a crepant resolution for some finite cover of $\overline{\0}$.
The following problem is then closely related to the focus of our paper.

\begin{problem}
Find out all symplectic covers of nilpotent orbit closures.
\end{problem}

Typical examples of symplectic covers are given by Springer maps. By \cite{Fu1}, crepant resolutions of nilpotent orbit closures are all given by Springer resolutions. It was observed in \cite{Fu2} that some covers of odd degrees do not admit any crepant resolution. Recently, Namikawa (\cite{N19}) constructed all $\mathbb{Q}$-factorial terminalizations of universal covers of nilpotent orbit closures in classical types. It is proven therein (\cite[Corollary 1.11]{N19}) that for nilpotent orbits in $\mathfrak{sl}_n$, symplectic covers are symplectic resolutions except for the double cover $\mathbb{C}^2 \to \overline{\0}_{[2]} \subset \mathfrak{sl}_2$. Starting from this, we can construct a new example of symplectic cover in type $C$ (\cite[Example 2.2]{N19}), which we now recall.

Let's consider the dual regular orbits $\cO_{[2n]} \subset \mathfrak{sp}_{2n}$ and $\cO_{[2n+1]} \subset \mathfrak{so}_{2n+1}$. They are Richardson orbits and have a unique polarization given by Borel subgroups. By Theorem \ref{t.Richardson}, the two orbit closures are mirrors from one to the other. Note that $\pi_1(\cO_{[2n]} )=\pi_1(\cO_{[2n+1]})=\mathbb{Z}_2$.  It is then natural to ask what is the mirror of the universal cover of $\overline{\cO}_{[2n]}$.

In fact, ${\cO}_{[2n]}$ has another induction as follows: let $Q \subset {\rm Sp}_{2n}$ be the parabolic subgroup with Levi subalgebra $\fl \cong \mathfrak{gl}_1^{\oplus n-1} \oplus \mathfrak{sp}_2$, then we have the following generalized Springer map
\begin{align*}
	{\rm Sp}_{2n} \times^Q (\fu + \overline{\cO}_{[2]}) \longrightarrow \overline{\cO}_{[2n]}.
\end{align*}
As  $\overline{\cO}_{[2]}= \mathbb{C}^2/ \pm 1$, we have the following diagram
\begin{align*}
	\xymatrix{{\rm Sp}_{2n} \times^Q (\fu + \mathbb{C}^2)  \ar[d]_{2:1} \ar[r]   & \widehat{\cO}_{[2n]}  \ar[d]^{2:1}  \\   {\rm Sp}_{2n} \times^Q (\fu + \overline{\cO}_{[2]}) \ar[r] & \overline{ \cO}_{[2n]} }
\end{align*}
here $\widehat{\cO}_{[2n]} $ is the universal cover of $\overline{ \cO}_{[2n]}$. This gives an example of symplectic cover, which is different from Springer maps. Note that the first horizontal map gives a crepant resolution of $\widehat{\cO}_{[2n]} $, which allows us to compute
the stringy E-polynomial of $\widehat{\cO}_{[2n]} $:
\begin{align*}
	{\rm E}_{\rm st} (\widehat{\cO}_{[2n]} ) &= {\rm E}({\rm Sp}_{2n} \times^Q (\fu + \mathbb{C}^2) ) \\
	                       &={\rm E}({\rm Sp}_{2n} \times^Q \fu  ) \times {\rm E}(\mathbb{C}^2) \\
	                       &={\rm E}(T^*({\rm Sp}_{2n}/Q)) \times {\rm E}(\mathbb{C}^2).
\end{align*}
By Proposition \ref{top.ms}, we have
\begin{align*}
	{\rm E}(T^*({\rm Sp}_{2n}/Q)) = {\rm E} (T^*({\rm SO}_{2n+1}/{}^LQ) )={\rm E}_{\rm st}(\overline{\cO}_{[2n-1, 1^2]}),
\end{align*}
where ${\cO}_{[2n-1, 1^2]}$ is the subregular orbit in $\mathfrak{so}_{2n+1}$.  This gives the following mirror pair:
\begin{proposition}
	The product $\overline{\cO}_{[2n-1,1^2]} \times \mathbb{C}^2$ is a mirror of the universal cover  $\widehat{\cO}_{[2n]}$ of $\overline{\cO}_{[2n]} \subset \mathfrak{sp}_{2n}$.\end{proposition}

It is unclear to us how to construct a mirror of the universal cover of  $\overline{\cO}_{[2n+1]}$.

We can generalize this construction by noting that the minimal nilpotent orbit closure $\overline{\0}_{min} \subset \mathfrak{sp}_{2q}$ is isomorphic to $\mathbb{C}^{2q}/\pm1$, which has a symplectic cover given by $\mathbb{C}^{2q} \to \overline{\0}_{min}$.

Take a parabolic subalgebra $\mathfrak{p} \subset \mathfrak{sp}_{2n}$ of type $(p_1, \cdots, p_l; 2q)$, namely its Levi part $\fl$ is given by
$$
\fl \cong \mathfrak{gl}_{p_1} \oplus   \mathfrak{gl}_{p_2} \oplus \cdots  \oplus \mathfrak{gl}_{p_l} \oplus \mathfrak{sp}_{2q},  \  \text{with} \ 2 \sum_{j=1}^l p_j + 2q = 2n.
$$

Let $P$ be a standard parabolic subgroup with Lie algebra $\mathfrak{p}$. Then we have the following induction
$$
{\rm Sp}_{2n} \times^P (\mathfrak{u} + \overline{\0}_{min}) \to \overline{\0}_{C},
$$
where $\0_{min}$ is the minimal nilpotent orbit in $\mathfrak{sp}_{2q}$. Using the double cover $\mathbb{C}^{2q} \to \overline{\0}_{min}$, we have a symplectic cover with its Stein factorization:
$$
{\rm Sp}_{2n} \times^P (\fu + \mathbb{C}^{2q})  \xrightarrow{{\text{crepant  resolution}}} \ X_C \xrightarrow{{\text{finite}}} \overline{\0}_C
$$

Now consider the Langlands dual parabolic subgroup ${}^LP$ and the associated Springer map $T^* ({\rm SO}_{2n+1}/{}^LP) \to \overline{\0}_B$. We denote by
$Y_B \to \overline{\0}_B$ the parabolic cover associated to ${}^LP$.  Then by a similar argument as before, we have
\begin{proposition}
The product $Y_B \times \mathbb{C}^{2q}$ is a mirror of $X_C$.
\end{proposition}

\begin{remark}
Note that $\mathfrak{b}_2=\mathfrak{c}_2$, hence the minimal nilpotent orbit closure $\overline{\0}_{min} \subset \mathfrak{so}_5$ is also isomorphic to $\mathbb{C}^4/\pm 1$, which again admits a double cover from $\mathbb{C}^4$.  Then we can repeat the previous construction to obtain more examples of mirror pairs.
\end{remark}

\section{Seesaw property of Langlands dual generalized Springer maps} \label{DGSM}

%\subsection{Induced orbits and Langlands duality} \label{IOLD}

By abuse of notations, we define the addition/subtraction of two partitions formally as follows:
\begin{align*}
	\mathbf{d}=[d_1, d_2, \cdots, d_n], \quad \mathbf{f}=[f_1, f_2, \cdots, f_m], \quad  \text{for} \quad n \geq m \\
	\mathbf{d} \pm \mathbf{f} := [d_1 \pm f_1, \cdots , d_m  \pm f_m, d_{m+1}, \cdots, d_n ].
\end{align*} 

Lusztig \cite{Lus79} shows that any nilpotent orbit induced from a special one is again special.
Now we will prove the converse as follows:

\begin{proposition} \label{special-from-rigid}
	Every special non-rigid orbit of type $B_n$ or $C_n$ can be induced from a rigid special orbit.
\end{proposition}
\begin{proof}
	We only prove it for type $C_n$ as the proof is similar for type $B_n$.
	Consider a special partition of type $C_n$:
	\begin{align*}
\mathbf{d} = [d_1 \equiv \cdots  \equiv d_{2k} \equiv 0 ,  d_{2k+1} \equiv d_{2k+2} \equiv 1, \ldots , d_{2j-1} \equiv d_{2j}, \ldots,  d_r], \text{with} \ d_{r} \geq 1.
\end{align*}

If all $d_i$ are even, then  $\0_{\mathbf{d}}$ is a Richardson orbit induced from the zero orbit.  In the following, we assume that there are odd parts in $\mathbf{d}$.
	
We may assume $d_r=1$ after subtracting an even partition from $\mathbf{d}$.  Write the last part of this partition as $[\ldots, d_s^{2n_s} \equiv 1, d_{s+1} \equiv \ldots \equiv d_{s+t} \equiv 0, 1^{2n_1}]$ (where $t=0$ means there is no even part between the two odd parts).  As $\mathbf{d}$ is special, the number $t$ of even parts between the two odd parts 1 and $d_s$ is even.  We define the rigid special partition
\begin{align*}
\mathbf{d}_{rs} =
\begin{cases}
    [2^{t}, 1^{2n_1}] \quad &\text{if} \quad t \geq 4    \\ [1^{t+2n_1}] \quad &\text{if} \quad t \leq 2,
\end{cases}
\end{align*}
which will induce the partition $[d_{s+1}, \ldots,  d_{s+t}, 1^{2n_1}]$.  Note that in the case $t=2$, we need to perform a $C$--collapse. We conclude the proof by repeating this process.
\end{proof}

\begin{remark}
An orbit is said {\em birationally rigid} if it is not birationally induced. By \cite[Proposition 2.3 ]{Los18},  any special orbit can be birationally induced from a birationally rigid special orbit, which is different from our claim, as a birationally rigid orbit is not necessarily rigid.
\end{remark}
\begin{example}
Given a special orbit, it can be induced from different rigid special orbits. For example,  in $\mathfrak{sp}_{66}$, the partition  $[8^4, 6^2, 4^4, 3^2]$ can be induced from $[2^{10}, 1^2] $ by our construction. On the other hand, it can also be induced from the following rigid special partition	
\begin{align*}
	\mathbf{d}_{rs}= [4^4, 3^2, 2^4, 1^2].
\end{align*}
The partition is obtained by a $C$-collapse of $[8^4, 7, 5, 4^4, 3^2]$.
\end{example}

\begin{theorem} \label{main-theorem-1}
Let $P<G$ be a parabolic subgroup with Lie algebra $\mathfrak{p}=\fl \oplus \fu$ and ${}^LP < {}^LG$ the Langlands dual parabolic subgroup with Lie algebra ${}^L\mathfrak{p}={}^L\fl \oplus {}^L\fu$. Let $\cO_{\mathbf{d}_{rs}} \subset \fl$ be a rigid special nilpotent orbit. Under the Springer duality and Langlands duality, we have the following diagram for generalized Springer maps:
\begin{align}
\begin{split}
\xymatrix{G\times^P(\mathfrak{u} + \overline{\mathcal{O}}_{\mathbf{d}_{rs}} )  \ar[d]_{\nu} & \stackrel{Langlands \ dual}{\rightsquigarrow } & {}^LG \times^{{}^L P} ({}^L\mathfrak{u} + \overline{{}^S\mathcal{O}}_{\mathbf{d}_{rs}} ) \ar[d]^{{}^L \nu}  \\   \overline{ \mathcal{O}}_{\bf d_C } & \stackrel{Springer\ dual}{\rightsquigarrow } & \widebar{ \mathcal{O}}_{\bf d_B }=\overline{{}^S \mathcal{O}}_{\bf d_C } }.	\label{duality-diagram-general}
\end{split}
\end{align}
Moreover, we have the following seesaw property for the degrees:
$$
\deg \nu \cdot \deg {}^L\nu = |\bar{A}({ \mathcal{O}}_{\bf d_C }) / \bar{A}(\cO_{\mathbf{d}_{rs}})| = |\bar{A}({ \mathcal{O}}_{\bf d_B }) / \bar{A}({}^S\cO_{\mathbf{d}_{rs}})|.
$$
\end{theorem}

Being induced orbits, there exists an even partition $\mathbf{Even}=[e_1, \cdots, e_m]$ (with $e_i \equiv 0, \forall i$) such that
 $\mathbf{d}_C=(\mathbf{Even} + \mathbf{d}_{rs})_C$ and $\mathbf{d}_B=(\mathbf{Even} + S(\mathbf{d}_{rs}))_B$. Let
	\begin{itemize}
		\item $c=$ number of collapses  $(\mathbf{Even} + \mathbf{d}_{rs} \rightarrow \mathbf{d}_C)$,
		\item $b=$ number of collapses $(\mathbf{Even} + S(\mathbf{d}_{rs}) \rightarrow \mathbf{d}_B)$.
	\end{itemize}
	
\begin{lemma} \label{l.GeneralizedSpringerMapDegree}
Under the above notations, we have
	\begin{align*}
		\deg \nu = 2^c, \quad \deg {}^L\nu = 2^b.
	\end{align*}
\end{lemma}
\begin{proof}
	Recall that  induction is transitive, i.e., if $\fl_1 \subset \fl_2 \subset \g$ are two Levi subalgebras, then
	\begin{align*}
		\operatorname{Ind}_{\fl_2}^{\g }(\operatorname{Ind}_{\fl_1}^{\fl_2}(\cO_{\fl_1})) = \operatorname{Ind}_{\fl_1}^{\g}({\cO_{\fl_1}}).
	\end{align*}
	Now we consider elementary inductions in the sense of Section 2.3 in \cite{Los18}. Let
	\begin{align*}
		\cO_{t}=(0, \cO_{\mathbf{d}}) \subset \mathfrak{gl}_{m} \oplus \mathfrak{so}_{2(n-m)+1} := \fl,
	\end{align*}
	where $\mathbf{d}=[d_1,\cdots, d_l]$. Then $\cO = \operatorname{Ind}_{\fl}^{\g}(\cO_t)$ has partition
	\begin{itemize}
		\item[(a)] either $[d_1+2, \cdots, d_m+2, d_{m+1}, \cdots,d_{l}]$, if this is a partition of type $B$,
		\item[(b)] or   $[d_1+2, \cdots, d_{m-1}+2 , d_m+1, d_{m+1}+1, d_{m+2} \cdots,d_{l}]$, otherwise.
	\end{itemize}
	
	Note that the induction of (a) gives a birational generalized Springer map by Claim (3.6.1) \cite{N19} (see also \cite[Theorem 7.1(d)]{He78}). Notice that (b) is the $B$--collapse of (a). Indeed, (b) happens exactly for $d_m=d_{m+1}$ is even and there exist an even number $i$ (may be 0) and an odd number $j$, such that.
	\begin{align*}
		d_{m-i}=\cdots=d_m=d_{m+1}=\cdots=d_{m+j}.
	\end{align*}
	Hence type (b) appears exactly when $\cO$ has a partition of the form
	\begin{align*}
		[d_1+2, \cdots, d_{m-1}+2, (d_{m}+1)^2, d_{m+2}, \cdots,d_l],
	\end{align*}
	i.e., there exists an odd part with multiplicity 2. This induction is of degree 2 by   \cite[Claim (3.6.2)]{N19} (see also \cite[Theorem 7.1(d)]{He78}).
	
	Similarly, for type $C$, (b) occurs only when $d_m=d_{m+1}$ is odd and there exist an even number $i$ (may be 0) and an odd number $j$, such that
	\begin{align*}
		d_{m-i}=\cdots=d_m=d_{m+1}=\cdots=d_{m+j}.
	\end{align*}
	Thus $\cO$ has partition $[d_1+2, \cdots, d_{m-1}+2, (d_{m}+1)^2, d_{m+2}, \cdots,d_l]$, and this induction gives a generalized Springer map  of degree 2 by  \cite[Claim (2.3.2)]{N19} (see also \cite[Theorem 7.1(d)]{He78}).
\end{proof}

Similar to the Richardson orbit, we have the following description of Lusztig's canonical quotient for the induced orbit: given an induction pair $(P, \cO_{\bf t})$, let $\cO= \Ind_{\fl}^{\g}(\cO_{\bf t}) $. Then
\begin{align*}
	 G \times^P (\fu + \overline{\cO}_{\bf t}) \longrightarrow  \overline{\cO}. 
\end{align*}
Here $\cO_{\bf t}$ could be any special orbit, not necessarily rigid special. 

\begin{lemma} \label{l.induceCanonicalQuotient_C}
Under the natural inclusion $\mathcal{A}(\cO_{\bf t}) \hookrightarrow \mathcal{A}(\cO) $, we have $\bar{A}(\cO_{\bf t})$ is a subgroup of $\bar{A}(\cO)$.
\end{lemma}
\begin{proof}
We prove for type $C$, as it is similar for type $B$. Without loss of generality, we assume the partition of $\cO_{\bf t}$ has the following form
\begin{align*}
	\mathbf{t} = [ \mathbf{t}_0, \mathbf{t}_1, \cdots, \mathbf{t}_{2r} ],
\end{align*} 
where
\begin{align*}
	& \mathbf{t}_j = [ t_{j,1}, \cdots, t_{j,N_j} ], \quad t_{j,1} \equiv \cdots \equiv t_{j,N_j} \equiv j, \quad \text{for} \quad 0 \leq j \leq 2r-1, \\
	& \mathbf{t}_{2r} = [ t_{2r,1} \equiv \cdots \equiv t_{2r, N_{2r}} \equiv 0]. 
\end{align*}
Since $\cO_{\bf t}$ is special, then $N_j \equiv 0$, for $0 \leq j \leq 2r-1$. 

Let $\mathbf{Even}$ be the partition induced from $P$ as in (\ref{induction}) of Section 2.3. Let
\begin{align*}
	\mathbf{d} = \mathbf{Even} + \mathbf{t} = [\mathbf{d}_0, \cdots, \mathbf{d}_{2r}], \quad \mathbf{d}_j = [ d_{j,1}, \cdots, d_{j,N_j} ].
\end{align*}
Then $\mathbf{d}_C$, the $C$--collapse of $\mathbf{d}$, is the partition of $\cO$. Recall $\mathcal{A}(\cO_{\bf t})$ is indexed by $B(\bf{t})= \{ t_{2j,i} \mid t_{2j,i} > t_{2j,i+1}, 0 \leq j \leq 2r \}$, and $\mathcal{A}(\cO)$ is indexed by $B(\bf{d}_C)$. Note that adding $\mathbf{Even}$ won't change the parity but may change the size relation of the adjacent parts. More preciously, if $t_{2j,i} > t_{2j,i+1}$, then $d_{2j,i} > d_{2j, i+1}$. Furthermore, $B(\bf{d}_C)$ may contain elements in $\mathbf{d}_{2m+1}$ for some $m$, since the $C$--collapse happens at $d_{2m+1, 2h+1} $ where $d_{2m+1, 2h+1} > d_{2m+1, 2h+2} $, and makes $d_{2m+1, 2h+1} \mapsto d_{2m+1, 2h+1}-1 \equiv 0 $, $d_{2m+1, 2h+2} \mapsto d_{2m+1, 2h+2}+1 \equiv 0 $. Then we conclude that $\mathcal{A}(\cO_{\bf t})$ is a subgroup of $\mathcal{A}(\cO)$.

From Lemma \ref{l.EqualityOfCanonicalQuotient} and Remark \ref{r.CombCanonicalQuotient}, we know that $\bar{A}(\cO_{\bf t})$ consists of $\{d_{2j, 2l} \}$ in $\mathbf{t}_{2j}$, for $0 \leq j \leq r$ and $1 \leq l \leq \left \lceil N_{2j}/2 \right \rceil -1 $, such that $ t_{2j, 2l} \neq t_{2j, 2l+1} $. Here $\left \lceil N_{2j}/2 \right \rceil$ means taking the smallest integer bigger than $N_{2j}/2$. Similarly, if $t_{2j, 2l} \neq t_{2j, 2l+1}$ in $\mathbf{t}_{2j}$, then $d_{2j, 2l} \neq d_{2j, 2l+1}$ in $\mathbf{d}_{2j}$, for $0 \leq j \leq r$. However, if $t_{2j, 2h}=t_{2j, 2h+1}$, for some $h$, it may happen that $d_{2j, 2h} > d_{2j, 2h+1}$, then $d_{2j, 2h}$ contributes to $\bar{A}(\cO)$. Besides, to obtain $\bar{A}(\cO_{\bf t})$, we remove all $\mathbf{t}_{2j+1}$, $0 \leq j \leq r-1$, since $t_{2j+1, 2l+1} = t_{2j+1, 2l+2}$ for all $0 \leq l \leq N_{2j+1}/2-1 $. However, after adding $\mathbf{Even}$, it may happen $d_{2j+1, 2h+1} > d_{2j+1, 2h+2}$ for some $h$. Then after $C$--collapse, $d_{2j+1, 2h+2}+1$ contributes to $\bar{A}(\cO)$. Finally, according to Remark \ref{r.CombCanonicalQuotient}, all the generators of $\bar{A}(\cO_{\bf t})$ contribute to $\bar{A}(\cO)$. Thus, we arrive at a conclusion.
\end{proof}

Let $\mathbf{d}=[d^\prime_1, d^\prime_2, \cdots ]$ be the induced partition as (\ref{induction}) and let
\begin{align*}
	I_{\chi}(\mathbf{d})=\{ j \in \mathbb{N} \mid j \equiv  d^{\prime}_j \equiv \chi,\  d^{\prime}_j \geq d^{\prime}_{j+1}+2 \}.
\end{align*}
The set $I_\chi(\mathbf{d})$ contains the parts where collapses ($\mathbf{Even}+\mathbf{t} \rightarrow \mathbf{d}_\varepsilon$ ) happen. Given an induction pair $(P, \cO_{\bf t})$ of type $C$. Let $\mathcal{A}_{P, \cO_{\bf t}}$ be a subgroup of $\mathcal{A}(\cO)$ consisting of $x=x_{i_k} + \cdots + x_{i_1} \in \mathcal{A}(\cO)$ where $x_{d_j}$ and $x_{d_{j+1}}$ appear simultaneously whenever $j \in I_{\chi}(\mathbf{d})$. Denote by $K(\cO)$ the kernel of the canonical quotient $\mathcal{A}(\cO) \rightarrow \bar{A}(\cO)$. Then we have

\begin{proposition} \label{p.induce-pair-group}
	$K(\cO)$ and $\bar{A}(\cO_{\bf t})$ are subgroups of $\mathcal{A}_{P, \cO_{\bf t}}$, and the similar results hold for the Springer dual orbit.
\end{proposition}
\begin{proof}
	The first claim is similar to Proposition \ref{p.A_P}, which we omit here. The second claim comes from the definition $\mathcal{A}_{P, \cO_{\bf t}}$ and description of $\bar{A}(\cO_{\bf t})$ in Lemma \ref{l.induceCanonicalQuotient_C}.
\end{proof}

Denote by $\bar{A}_{P, \cO_{\bf t}}$ the quotient $\mathcal{A}(\cO) / \mathcal{A}_{P, \cO_{\bf t}}$, by Lemma \ref{l.induceCanonicalQuotient_C} and Proposition \ref{p.induce-pair-group}, it is a subgroup of $\bar{A}(\cO)/ \bar{A}(\cO_{\bf t})$. Similar notations apply for the dual induced orbit, i.e., $\mathcal{A}({}^S\cO)$, $\bar{A}_{{}^LP, {}^S\cO_{\bf t}}$, etc. Then we have the following 
%a seesaw phenomenon for Lusztig's canonical quotient for induced orbits:
\begin{proposition} \label{p.seesaw}
	Under the isomorphism $\bar{A}(\cO)/ \bar{A}(\cO_{\bf t}) \cong  \bar{A}({}^S\cO)/ \bar{A}({}^S\cO_{\bf t})$, we have $\bar{A}(\cO)/ \bar{A}(\cO_{\bf t}) \cong \bar{A}_{P, \cO_{\bf t}} \times \bar{A}_{{}^LP, {}^S\cO_{\bf t}}$.
\end{proposition}
\begin{proof}
	The isomorphism comes from Lemma \ref{l.EqualityOfCanonicalQuotient} and Lemma \ref{l.induceCanonicalQuotient_C}.  Without loss of generality, we assume $\cO$ is of type $C$, and we follow the notations in the proof of Lemma \ref{l.induceCanonicalQuotient_C}, i.e., the partition of $\cO$ is given by the $C$--collapse of
	\begin{align*}
	    \mathbf{d} = \mathbf{Even} + \mathbf{t} = [\mathbf{d}_0, \cdots, \mathbf{d}_{2r}], \quad \mathbf{d}_j = [ d_{j,1}, \cdots, d_{j,N_j} ].
	\end{align*}
	Let
	\begin{align*}
		I_{\rm odd} &= \{ d_{2j+1, 2h+2} \mid d_{2j+1,2h+1} > d_{2j+1,2h+2}, \ \text{for} \ 0 \leq j \leq r-1,\ 0 \leq h \leq N_{2j+1}/2-1 \}, \\
		I_{\rm even} &= \{ d_{2j, 2h} \mid d_{2j,2h} > d_{2j,2h+1}, \ \text{for} \ 1 \leq j \leq r,\ 1 \leq h \leq \left \lceil N_{2j}/2 \right \rceil -1 \}.
	\end{align*}
	Here $\left \lceil N_{2j}/2 \right \rceil$ means taking the smallest integer bigger than $N_{2j}/2$. From the proof of Lemma \ref{l.induceCanonicalQuotient_C}, the generators of $\bar{A}(\cO)/ \bar{A}(\cO_{\bf t})$ are indexed by $I_{\rm odd}$ and $I_{\rm even}$. By the definition of $\mathcal{A}_{P, \cO_{\bf t}}$, it is easy to obtain that the generators of $\bar{A}_{P, \cO_{\bf t}}$ are indexed by $I_{\rm odd}$, and the generators of $\bar{A}_{{}^LP, {}^S\cO_{\bf t}}$ are indexed by $I_{\rm even}$, which concludes the proof.
	
\end{proof}

Then we conclude that Proposition \ref{p.inducedCanonicalQuotient} follows from Lemma \ref{l.induceCanonicalQuotient_C}, Proposition \ref{p.induce-pair-group} and \ref{p.seesaw}. \\

Now we can prove Theorem \ref{main-theorem-1}. \\

\begin{proof}
	By \cite[Section 3]{L-S},  the induction behaves well under Langlands duality, which shows the first part of the claim. The second part comes from Lemma \ref{l.GeneralizedSpringerMapDegree} and Proposition \ref{p.seesaw}.
\end{proof}

Now we conclude the proof of Theorem \ref{t.induced}: (1) follows from Proposition \ref{canonical-quotient-rigid-special}, (2) follows from Lemma \ref{l.EqualityOfCanonicalQuotient}, while  (3) and (4) follow from Theorem \ref{main-theorem-1}.

\section{Asymmetry of footprint}

In this section, we investigate the importance of the range of degrees of (generalized) Springer maps, which we call footprint. We first show how to obtain the footprint combinatorially for Richardson orbits and then provide a counter-example to show that the mirror symmetry may fail outside the footprint. Finally, we produce another example of rigid special orbits to show the asymmetry in a broader situation, which motivates our formulation of the mirror symmetry conjecture (Conjecture \ref{conj}.) for special orbits.

\subsection{Footprint: the range of degrees of Springer maps}

Recall that standard parabolic subgroups in $G$ are in bijection with subsets of simple roots. We denote by $P_{\Theta}$ the standard parabolic subgroup whose Levi subalgebra is generated by $\g_\alpha$,  with $\alpha \in \mathfrak{h} \cup \langle \Theta \rangle$.
Two standard parabolic subgroups $P_{\Theta}$ and $P_{\Theta^{\prime}}$ are said \emph{$R$--equivalent}, denoted by $P_{\Theta} \sim_R P_{\Theta^{\prime}}$, if $P_{\Theta}$ and $P_{\Theta^{\prime}}$ are both polarizations of the same Richardson orbit.  The following result from \cite[Theorem 3]{Hi81} gives  the fundamental relations for the $R$--equivalence in case of type $B_n$ or $C_n$:
\begin{proposition} \label{p.Hirai}
Let $P_k$ be the maximal parabolic subgroup associated with the $k$--th simple root. Assume $\g$ is of type $B_n$ or $C_n$.  Then the $R$--equivalence is generated by the following:

(i) in $A_\ell$, $P_k \sim_R P_{\ell+1-k}$;

(ii) in $B_{3k-1}$ or $C_{3k-1}$, $P_{2k-1} \sim_R P_{2k}$.
\end{proposition}

The degree of any Springer map in classical types is computed in \cite{He78}, from which we can deduce that  (see also \cite[Lemma 5.4]{Fu4}), for type $B_{3k-1}$ (resp. $C_{3k-1}$), the Springer map associated to $P_{2k-1} $ (resp. $P_{2k}$) is birational, and that associated to  $P_{2k}$ (resp. $P_{2k-1}$) is of degree 2.  By \cite[Proposition 5.7]{Fu4}, the degree of the Springer map remains the same if we perform the equivalence (i), and it doubles (or reduces to half) if we perform the equivalence (ii). It follows that for any Richardson orbit $\0$ of type $B$ or $C$,  there exist some non-negative integers $m_{in}, m_{ax}$ (depending on $\0$) such that
$$
\{ {\rm deg}(\nu_P) | P \in {\rm Pol}(\0)\} = \{2^{m_{in}}, 2^{m_{in}+1}, \cdots, 2^{m_{ax}}\}.
$$

It remains to determine the integers $m_{in}, m_{ax}$ in terms of the partition ${\bf d}$ of $\0$.
Consider first the case of type $C$.
Being a Richardson orbit, its partition has the following form
 \begin{align*}
\mathbf{d}_C = [\underbrace{d_1\equiv d_2, \cdots, d_{2k-1} \equiv  d_{2k}\equiv 1}_{\text{if}\  d_{2j} \equiv d_{2j+1} \equiv 0, \text{ then} \  d_{2j} \geq d_{2j+1} +2}, d_{2k+1} \equiv \ldots \equiv  d_{2r+1}\equiv 0 ],
\end{align*}
where $d_{2r} \geq 2$ and $d_{2r+1}$ may be zero.

Define the following two non-negative integers:
\begin{itemize}
\item $2\alpha=$ number of even parts in $d_1, d_2, \cdots, d_{2k}$.
\item $m$ = the unique number   such that
\begin{itemize}
\item For $0 \leq j < m $,	 $d_{2k+2j+1}\geq d_{2k+2j+2} > d_{2k+2j+3}$,
\item For $j=m$,	 $d_{2k+2m+1}\geq d_{2k+2m+2} = d_{2k+2m+3}$, or $d_{2k+2m+1}=d_{2r+1}$.
\end{itemize}
%and if $k=r$, set $m=0$.	
\end{itemize}

\begin{lemma} \label{C-R}
Under previous notations, we have $m_{in} = \alpha$ and $m_{ax} = \alpha+m$. In other words,
$$
\{ {\rm deg}(\nu_P) | P \in {\rm Pol}(\0)\} = \{2^\alpha, 2^{\alpha+1}, \cdots, 2^{\alpha+m}\}.
$$
\end{lemma}
\begin{proof}
     Let $P$ be a polarization of $\cO$ with Levi type $(p_1, \ldots, p_l ;2q)$, then the induced partition is given by 
	\begin{align*}
		\mathbf{d}(P) = [d_1^{\prime} \equiv \dots \equiv d_{2q}^{\prime} \equiv 1,  d_{2q+1}^{\prime} \equiv \dots \equiv d_{2s+1}^{\prime}\equiv 0 ].
	\end{align*}
	As $P$ is a polarization, we have $\mathbf{d}_C = (\mathbf{d}(P))_C$.  Note that $C$-collapse happens only among the first $2q$ parts of $\mathbf{d}_C$, and each $C$-collapse introduces two even numbers. It follows that we need to perform at least $\alpha$ collapses to arrive at $\mathbf{d}_C$, which shows that $\deg(\nu_P) \geq 2^\alpha$ by Lemma \ref{l.GeneralizedSpringerMapDegree}.

Let $I = \{1 \leq j \leq k-1| d_{2j-1} \equiv d_{2j} \equiv 0\}$, which has cardinality $\alpha$.  We define
$\mathbf{d'}$ as follows:
\begin{align*}
d'_{2j-1} = d_{2j-1}+1, d'_{2j} = d_{2j}-1, \text{if } \ j \in I, \quad  \text{and} \ d'_i = d_i \ \text{otherwise}.
\end{align*}

Note that if $d_{2j} \equiv d_{2j+1} \equiv 0$, then  $d_{2j} \geq d_{2j+1} +2$, hence $\mathbf{d'}$ is again a partition. Moreover
$d'_1 \equiv \cdots \equiv d'_{2k} \equiv 1$ and $d_{2k+1} \equiv \cdots \equiv d_{2n} \equiv 0$.  It follows that there exists a parabolic subgroup $P' \in {\rm Pol}(\0)$ with $\mathbf{d'}$ being the induced partition. As $\mathbf{d}_C$ is obtained from $\mathbf{d'}$ by performing $\alpha$ $C$-collapses, we have 
$\deg(\nu_{P'})=2^\alpha$ by Lemma \ref{l.GeneralizedSpringerMapDegree}, hence $m_{in} = \alpha$.

	If $d_{2k+1} \geq d_{2k+2} > d_{2k+3}$,  we can obtain $\mathbf{d'}$ by one collapse from the following partition
	\begin{align*}
		\mathbf{d}^{\prime\prime} = [ d_1^{\prime} \equiv \dots \equiv d_{2k}^{\prime} \equiv d_{2k+1}^{\prime}+1 \equiv d_{2k+2}^{\prime}-1 \equiv 1,  d^{\prime}_{2k+3} \equiv \dots \equiv d_{2r+1}^{\prime}\equiv 0].
	\end{align*}
The corresponding parabolic subgroup $P^{\prime\prime}$ satisfies $\deg(\nu_{P^{\prime\prime}})=2 \deg(\nu_{P'}) = 2^{\alpha+1}$.
We can repeat this construction to arrive at a parabolic subgroup whose Springer map has degree $2^{\alpha+m}$, which is the maximal degree. 
\end{proof}

The partition  $\mathbf{d}_B= S(\mathbf{d}_C)$ of the Springer dual orbit ${}^S \cO$ is of the following form (by abuse of notation)
 \begin{align*}
\mathbf{d}_B=[d_1 \equiv d_2 \equiv \cdots \equiv d_{2l-1} \equiv 1 ,  \underbrace{d_{2l} \equiv d_{2l+1}\equiv 0, \ldots, d_{2j} \equiv d_{2j+1}, \cdots, d_{2n+1}}_{\text{if}\  d_{2j-1} \equiv d_{2j} \equiv 1, \text{ then} \  d_{2j-1} \geq d_{2j} +2}].
\end{align*}

We define the following two non-negative integers:
\begin{itemize}
\item $2\beta=$ number of odd parts in $d_{2l+2}, \cdots, d_{2n+1}$.
\item $m^\prime$ = the unique number  such that
\begin{itemize}
\item For $0 \leq j < m^\prime $,	 $d_{2l-2j-1} \leq d_{2l-2j-2} < d_{2l-2j-3}$,
\item For $j=m^\prime$,	 $d_{2l-2m^\prime-1} \leq d_{2l-2m^\prime-2} = d_{2l-2m^\prime-3}$, or $d_{2l-2m^\prime-1}=d_{1}$.
\end{itemize}
\end{itemize}

By a similar argument as that for Lemma \ref{C-R}, we have
\begin{lemma} \label{B-R}
	Under previous notations, we have $m_{in} = \beta$ and $m_{ax} = \beta+m' $.   In other words,
$$
\{ {\rm deg}(\nu_P) | P \in {\rm Pol}(\0)\} = \{2^\beta, 2^{\beta+1}, \cdots, 2^{\beta+m'}\}.
$$
\end{lemma}

Note that by Corollary \ref{c.DegreeDuality}, we have $m=m'$. With the above notations and Lemmas, we obtain the following result on the degree range  

\begin{proposition} \label{degree-range}
The set of  degrees of the parabolic cover $\mu_P: X_P \to \overline{\0}$ and its Springer dual $\mu_{{}^LP}: X_{{}^LP} \to \overline{{}^S\0}$ are given by the following
\begin{align*}
	\{(\deg \mu_P, \deg \mu_{{}^L P}) | P \in {\rm Pol}(\0)\} = \{(2^\beta, 2^{\alpha+m}), (2^{\beta+1}, 2^{\alpha+m-1}), \cdots, (2^{\beta+m}, 2^\alpha)\}.
\end{align*}
Moreover, we have $\widebar{A}(\cO)\simeq \widebar{A}({}^S\cO) \simeq \mathbb{Z}_2^{\alpha+\beta+m}$.
\end{proposition}

\begin{remark}
Here is a simple example showing that the ranges of $\deg \mu_P$ and $\deg \mu_{{}^L P}$ are NOT symmetric, i.e., $\alpha \neq \beta$ in general.  The orbit $\0_C = \0_{[4^2,1^2]} \subset \mathfrak{sp}_{10}$ has $c=1$ with  a polarization of degree 2, while its dual orbit $\0_B=\0_{[5,3,1^3]} \subset \mathfrak{so}_{11}$ has $b=0$ with  a birational polarization.
\end{remark}

\subsection{Stringy E-functions of special spherical nilpotent orbits}\label{sph-orb}

\subsubsection{Log resolutions for special spherical nilpotent orbits}

For a nilpotent element $x\in \mathfrak{g}$, take an $\mathfrak{sl}_2$--triplet $(x, y, h)$. Then we can write $\mathfrak{g}$ as an $\mathfrak{sl}_2$--module
$$
\mathfrak{g} = \oplus_{j=-m}^m \mathfrak{g}_j, \ \text{with} \  \g_j=\{z \in \g \mid [h, z] = j z\} \  \text{and} \ \mathfrak{g}_{m} \neq 0.
$$
This number $m$ is called the height of the nilpotent orbit $\0_x = G \cdot x$, which can be computed directly from its weighted Dynkin diagram as follows.  If we denote by $c_i$ (resp. $c'_i$) the coefficient of the highest root (resp. the weight of the marked Dynkin diagram of $\0$) at the $i$--th simple root, then the height of $\0$ is given by $\sum_{i} c_i c'_i$.

Recall that a nilpotent orbit $\0 \subset \mathfrak{g}$ is said {\em spherical} if $\0$ contains an open $B$--orbit. It is shown in  \cite{Pan94} that a nilpotent orbit $\0$ is spherical if and only if its height is at most 3. In classical Lie algebras, spherical nilpotent orbits can be reformulated as
\begin{itemize}
\item[(1)] For $\mathfrak{sl}$ or $\mathfrak{sp}$,  $\0_x$ is spherical iff $x^2=0$.
\item[(2)] For $\mathfrak{so}$, $\0_x$ is spherical iff ${\rm rk}(x^2) \leq 1$.
\end{itemize}

As a corollary, a spherical nilpotent orbit in $\mathfrak{sp}_{2n}$ (resp. $\mathfrak{so}_{2n+1}$) corresponds to a partition of the form $[2^k,1^{2n-2k}]$ (resp. $[3,2^{2k-2}, 1^{2n+2-4k}]$) for some $k$. Restricting to special nilpotent orbits, we have

\begin{lemma}
For any two integers $r, l \geq 1$, the following  two orbits are special and spherical nilpotent orbits of types $B$ and $C$:
$$
\0_C:=\0_{r, l}^C := \0_{[2^{2r},1^{2l}]} \subset \mathfrak{sp}_{4r+2l}, \quad \quad  \0_B:=\0_{r, l}^B = \0_{[3,2^{2r-2},1^{2l+2}]} \subset \mathfrak{sp}_{4r+2l+1}.
$$
The two orbits have normal closures, and they are Springer dual and of dimension $4r^2+4rl+2r$.  They are rigid if and only if $r \geq 2$, and in this case, we have
$$
\pi_1(\0_B)=\bar{A}(\0_B) \simeq \bar{A}(\0_C) = \pi_1(\0_C) = \mathbb{Z}_2.
$$
\end{lemma}
%\begin{remark}
%If $l=0$, then both orbits are Richardson which .
%\end{remark}

The following observation (and its proof) is due to Paul Levy, which gives a generalization of the shared orbits of \cite{BK94} for the pair $(\mathfrak{so}_{2n+1}, \mathfrak{so}_{2n+2})$.
\begin{proposition} \label{p.levy}
Consider the following nilpotent orbit
$$
\0_D:=\0_{r, l}^D = \0_{[2^{2r},1^{2l+2}]} \subset \mathfrak{so}_{4r+2l+2}.
$$
Then there exists an ${\rm SO}_{4r+2l+1}$--equivariant double cover
$\overline{\0}_{r,l}^D \to \overline{\0}_{r,l}^B$.
\end{proposition}
\begin{proof}
Write an element $X$ of $\overline{\0}_{r, l}^D$ as
$$
\begin{pmatrix}  M & u\\
-u^T & 0 \end{pmatrix}
$$
where $M \in \mathfrak{so}_{4r+2l+1}$ and $u$ is a column vector. As $X \in \overline{\0}_{[2^{2r},1^{2l+2}]}$, we have $X^2=0$. This implies that $M^2=uu^T$ and  $Mu=0$, thus ${\rm rk}(M^2) \leq 1$ and $M^3=0$.  Moreover we have ${\rm rk M} \leq  {\rm rk X} \leq 2r$.  It follows that $M \in \overline{\0}_{r,l}^B$, which gives the ${\rm SO}_{4r+2l+1}$-equivariant map $\overline{\0}_{r,l}^D \to \overline{\0}_{r,l}^B$. As both have the same dimension, this is a surjective map. Note that the fiber over a general point $M$ has two points corresponding to $(M, u)$ and $(M, -u)$, where $u$ is determined (up to a sign) by $M^2=uu^T$, hence this is a double cover.
\end{proof}

Now in order to compute the stringy E-functions, we will construct log resolutions, which are miraculously related to the classical varieties of complete quadrics and skew forms (cf. \cite{Be97}, \cite{Th99}). Let us first recall the Jacobson-Morosov resolution associated to an $\mathfrak{sl}_2$-triplet and the weight decomposition $\g= \oplus_{j=-m}^m \g_j$. Consider the  parabolic subalgebra $\p = \oplus_{j \geq 0} \g_j$,  which corresponds to $\p_{\Theta}$ with $\Theta \subset \Delta$ being the set of simple roots where the weighted Dynkin diagram of $\0$ has non-zero weight. Let $\n = \oplus_{i \geq 2} \g_i$, which is a $\p$-module.  Let $P \subset G$ be a parabolic subgroup with Lie algebra $\p$.  The Jacobson-Morosov resolution is given by
\begin{align*}
	G \times^P \n \longrightarrow \overline{\cO}.
\end{align*}
Let $G_C := {\rm Sp}_{2n}$ and $G_D := {\rm SO}_{2n+2}$, then we have two resolutions:
\begin{align*}
	 G_C \times^{P_C} \n_C  \longrightarrow \overline{\cO}_C, \quad
	 G_D \times^{P_D} \n_D  \longrightarrow \overline{\cO}_D.
\end{align*}

The following result is from \cite[Table 1]{BP19}, and we reproduce the proof here for the reader's convenience.

\begin{proposition}
	Under the action of $P_C$ (resp. $P_D$), $\n_C$ (resp. $\n_D$) becomes an ${\rm SL}_{2r}$--module. Moreover,  there exist two vector spaces   $V_C \simeq V_D \simeq \mathbb{C}^{2r}$ such that
	\begin{align*}
		\n_C = \mathrm{Sym}^2 V_C  \quad \n_D = \wedge^2 V_D.
	\end{align*}
\end{proposition}
\begin{proof}
Consider first the case $\0_C$, whose weighted Dynkin diagram has weight 1 at the $2r$--th node and zeros elsewhere. It follows that
$$
\mathfrak{n}_C = \mathfrak{g}_2^C = \oplus_{\alpha \in I} \mathfrak{g}_\alpha^C, \ {\text with} \ I=\{2\varepsilon_i^C, \varepsilon_i^C+\varepsilon_j^C, 1\leq i < j \leq 2r\}.
$$
It follows that $\mathfrak{g}^C_2 \simeq {\rm Sym}^2 V_C$, where $V_C = \oplus_{j=1}^{2r} \mathfrak{g}_{\varepsilon_j}^C$ is a vector space of dimension $2r$. The parabolic subgroup $P_C$ is the maximal parabolic associated with the $2r$--th root. Note that $[\mathfrak{u}_C, \mathfrak{n}_C]=0$ for the nilradical $\mathfrak{u}_C$ of $\mathfrak{p}_C$, so the unipotent part of $P_C$ acts trivially on $\mathfrak{g}_2^C$. The Levi part of $P$ is $L_C={\rm SL}_{2r} \oplus {\rm Sp}_{2l}$, with ${\rm Sp}_{2l}$ acts trivially on $\mathfrak{g}_2^C$ while ${\rm SL}_{2r}$ acts on $\mathfrak{g}_2^C \simeq {\rm Sym}^2 V_C$ in the natural way.

Similarly, for $\0_D$, its weighted Dynkin diagram has weight 1 at the $2r$--th node and zeros elsewhere.
This implies that
$$
\mathfrak{n}_D = \mathfrak{g}_2^D = \oplus_{\alpha \in J} \mathfrak{g}_\alpha^D, \ {\text with} \ J=\{\varepsilon_i^D+\varepsilon_j^D, 1\leq i < j \leq 2r\}.
$$
It follows that $\mathfrak{n}_D  \simeq \wedge^2 V_D$, where $V_D= \oplus_{j=1}^{2r} \mathfrak{g}_{\varepsilon_j}^D$ is a vector space of dimension $2r$. The remaining proof is similar to the type $C$ case.
\end{proof}

The Jacobson-Morosov resolution is generally not a log resolution, but we will construct a log resolution from it by successive blowups. For $i=0, 1, \cdots, 2r$, let us denote by $O^C_i \subset {\rm Sym}^2 V_C$ the set of elements of rank $i$ (viewing ${\rm Sym}^2 V_C$ as symmetric matrices of size $2r \times 2r$) ,  and denote by $\bar{O}^C_i$ its closure. Note that both $O^C_i$ and $\bar{O}^C_i$ are $P_{C}$-invariant. We remark that the Jacobson-Morosov resolution induces the following map on each stratum:
\begin{align*}
	G_C \times^{P_C} O_k^C \longrightarrow \0_k^C:=\cO_{[2^k, 1^{4r+2l-2k}]} \subset \overline{\0}_C= \bigcup_{k=0}^{2r} \0_k^C,\  k=0, \cdots, 2r.
\end{align*}

Consider the following birational map
$$
\phi: \widehat{\mathfrak{n}}_C \to \mathfrak{n}_C = {\rm Sym}^2 V_C
$$
obtained by successive blowups of $\mathfrak{n}_C =  {\rm Sym}^2 V_C$ along strict transforms of $\bar{O}^C_i$ from smallest $O^C_0 $ to the biggest $\bar{O}^C_{2r-2}$ (note that $\bar{O}^C_{2r-1}$ is the determinant hypersurface). As all are ${\rm SL}_{2r}$--equivariant, this gives the following birational map.
$$
\Phi:  \widehat{Z}_C:=G_C \times^{P_C} \widehat{\mathfrak{n}}_C \to Z_C:=G_C \times^{P_{C}} \mathfrak{n}_C \to  \overline{\0}_{r,l}^C.
$$
Let us denote by $\D_i^C$ the exceptional divisor of $\Phi$  over $\overline{\0}_i^C$ for $i=0, \cdots, 2r-1$.
Let us denote by $D_{2r-1}:=G_C \times^{P_C} \bar{O}_{2r-1}^C$, which is the exceptional divisor in $Z_C$.
Note that $\D_{2r-1}^C \to D_{2r-1}$ is birational.

\begin{proposition} \label{p.logresolutionC}
The morphism $\Phi$ is a log resolution for $\overline{\0}_{r,l}^C$, and we have
$$
K_{\widehat{Z}_C} = 2l  \D^C_{2r-1} + \sum_{j=0}^{2r-2} {\big (}\cfrac{(2r-j)(2r+1-j)}{2}-1 {\big )} \D^C_j.
$$
\end{proposition}
\begin{proof}
The first claim follows from classical constructions (\cite{Be97, Th99}). As each step is a blowup along a smooth center, the discrepancy of $\D_j^C$ is given by ${\rm codim}_{O_i}(\mathfrak{n}_C)-1$ for $i \leq 2r-2$, which is $\frac{(2r-j)(2r+1-j)}{2}-1$. For the discrepancy of $\D^C_{2r-1}$, we note that the isotropic Grassmannian $G_C/P_{C}$ has index $2r+2l+1$, while the first Chern class of $G_C \times^{P_{C}} \mathfrak{n}_C$ (viewed as a vector bundle over $G_C/P_{C}$) is given by the negative sum of roots defining $\mathfrak{n}_C$, which is $-(2r+1)(\varepsilon_1^C + \cdots + \varepsilon_{2r}^C) = -(2r+1) \lambda_{2r}^C$, where $\lambda_{2r}^C$ is the $2r$--th fundamental weight. By adjunction formula, this gives that $K_{Z_C} = 2l  D_{2r-1}$, hence the discrepancy of $\D^C_{2r-1}$ is $2l$.
\end{proof}

Now we turn to the case of $\0_D$. For $i=0, 1, \cdots, r$, let us denote by $O^D_{2i} \subset \wedge^2 V_D$ the set of elements of rank $2i$ (viewing $\wedge^2 V_D$ as skew-symmetric matrices of size $2r \times 2r$) ,  and denote by $\bar{O}^D_{2i}$ its closure. Note that both $O^D_{2i}$ and $\bar{O}^D_{2i}$ are $P_{D}$-invariant. Again the Jacobson-Morosov resolution induces
\begin{align*}
	G_D \times^{P_D} O_{2k}^D \longrightarrow \0_{2k}^D:=\cO_{[2^{2k}, 1^{4r+2l+2-4k}]} \subset \overline{\0}_D.
\end{align*}

Consider the following birational map
$$
\psi: \widehat{\mathfrak{n}}_D \to \mathfrak{n}_D = \wedge^2 V_D
$$
obtained by  successive blowups of $\mathfrak{n}_D =  \wedge^2 V_D$ along strict transforms of $\bar{O}^D_{2i}$ from smallest $O^D_0$ to the biggest $\bar{O}^D_{2r-4}$ (note that $\bar{O}^D_{2r-2}$ is the determinant hypersurface). This gives the following birational map.
$$
\Psi:  \widehat{Z}_D:=G_D \times^{P_D} \widehat{\mathfrak{n}}_D \to Z_D:=G_D \times^{P_{D}} \mathfrak{n}_D \to  \overline{\0}_{r,l}^D.
$$
Let us denote by $\D_{2i}^D$ the exceptional divisor over $\overline{\0}_{2i}^D$ for $i=0, \cdots, r-1$.  Note that $\D_{2r-2}^D \to  G_D \times^{P_D} \bar{O}_{2r-2}^D$ is birational.

The proof of the following result is similar to that of Proposition \ref{p.logresolutionC}. We just need to note that the orthogonal Grassmannian $G_D/P_{D}$ has the same index $2r+2l+1$ and  the negative sum of roots defining $\mathfrak{n}_D$ is $-(2r-1)(\varepsilon_1^D + \cdots + \varepsilon_{2r}^D) = -(2r-1) \lambda_{2r}^D$, where $\lambda_{2r}^D$ is the $2r$-th fundamental weight.

\begin{proposition}\label{p.logresolutionD}
The morphism $\Psi$ is a log resolution for $\overline{\0}_{r,l}^D$ and we have
$$
K_{\widehat{Z}_D} = (2l+2)  \D^D_{2r-2} + \sum_{j=0}^{r-2} ((r-j)(2r-2j-1)-1) \D^D_{2j}.
$$
\end{proposition}

\subsubsection{Asymmetry for spherical Richardson orbits}
We consider the case  $r=1$ in this subsection,  namely we have $\cO_C=\cO_{[2^2, 1^{2l}]} \subset \mathfrak{sp}_{2l+4}$ and $\cO_B = \cO_{[3, 1^{2l+2}]} \subset \mathfrak{so}_{2l+5}$.
Motivated by Question \ref{question}, we are going to compute the stringy E-functions of $\overline{\cO}_C$ and $\overline{\cO}_D$ ( which is the universal cover of $\overline{\cO}_B$).

\begin{proposition} \label{p.Omin}
Let $\0_{min} \subset \g$ be the minimal nilpotent orbit in a simple Lie algebra and $2d = \dim \0_{min}$.
Then
$$
{\rm E}_{\rm st}(\overline{\0}_{min}) = {\rm E} (\mathbb{P} \0_{min}) \times \cfrac{(q-1)q^d}{(q^d-1)}.
$$
\end{proposition}
\begin{proof}
Let $Z \to \overline{\0}_{min}$ be the blowup of the origin, which is a log resolution (cf. \cite[(2.6)]{Be00}).
This is, in fact, the Jacobson-Morosov resolution.
  The exceptional divisor $\D$
is $\mathbb{P} \0_{min}$ and $Z$ is the total space of the line bundle $\0_{\mathbb{P} \g}(-1)|_{\mathbb{P} \0_{min}}$.
By  \cite[(3.5)]{Be00}, we have $K_Z = (d-1) \D$. We can now compute
$$
{\rm E}_{\rm st} (\overline{\0}_{min}) = {\rm E}(\0_{min}) + {\rm E}(\D) \cfrac{q-1}{q^{d}-1} ={\rm E} (\mathbb{P} \0_{min}) \times \cfrac{(q-1)q^d}{(q^d-1)}.
$$
\end{proof}

By a direct computation using Propositions \ref{p.EPolyGr} and \ref{p.Omin}, we have

\begin{proposition} \label{p.r=1D}
Let $\overline{\0}_D = \overline{\0}_{[2^2,1^{2l+2}]} \subset \mathfrak{so}_{2l+6}$ be the minimal nilpotent orbit closure (of dimension $4l+6$). Then its stringy E-function is given by
$$
{\rm E}_{\rm st} (\overline{\0}_D) = \cfrac{(q^{l+1}+1)(q^{l+3}-1)(q^{2l+4}-1)q^{2l+3}}{(q^2-1)(q^{2l+3}-1)}.
$$
\end{proposition}

\begin{remark}
In a similar way, we can compute ${\rm E}_{\rm st}(\overline{\0}_{min})$ for all other classical types, which give the following (note that all are polynomials):
\begin{align*}
 A_n: & \   {\rm E}_{\rm st}(\overline{\0}_{min}) ={\rm E}(T^* \mathbb{P}^n) = \cfrac{(q^{n+1}-1)q^n}{q-1}. \\
 B_n: &  \ {\rm E}_{\rm st}(\overline{\0}_{min}) =\cfrac{(q^{2n}-1)q^{2n-2}}{q^2-1}.\\
 C_n: & \  {\rm E}_{\rm st}(\overline{\0}_{min}) ={\rm E}_{\rm st} (\mathbb{C}^{2n}/\pm 1) = q^{2n}+q^n.
\end{align*}

\end{remark}

\begin{proposition}\label{p.r=1C}
Let $\overline{\0}_C = \overline{\0}_{[2^2,1^{2l}]} \subset \mathfrak{sp}_{2l+4}$ be the minimal special nilpotent orbit closure  (of dimension $4l+6$). Then its stringy E-function is given by
$$
{\rm E}_{\rm st} (\overline{\0}_C) = \frac{(q^{2l+2}-1)(q^{2l+3}-1)(q^{2l+4}-1)q^3}{(q^2-1)(q^3-1)(q^{2l+1}-1)}.
$$
\end{proposition}
\begin{proof}
In this case, $\mathfrak{n}_C \simeq {\rm Sym}^2(\mathbb{C}^2)$ is the set of symmetric matrices of size $2 \times 2$ and $\mathfrak{n}_C = O_2 \bigsqcup O_1 \bigsqcup O_0$, where $O_0$ is the origin and $O_1$ is the affine quadric surface $\hat{\mathbb{Q}}^2$ minus the origin. 
It follows that ${\rm E}(O_1)={\rm E}(\mathbb{P}^1) \times {\rm E}(\mathbb{C}^*) = q^2-1$ and ${\rm E}(O_2) = q^2(q-1)$.
By Proposition \ref{p.logresolutionC}, the log resolution $\Phi: \widehat{Z} \to \overline{\0}^C_{1,l}$ is obtained by blowing up the zero section in the Jacobson-Morosov resolution and it  satisfies
$$
K_{\widehat{Z}} = 2l \D_1 + 2 \D_0.
$$

Note that $\D_1 = G_C \times^{P_C} \tilde{\mathbb{Q}}^2 $ and $\D_0 = G_C \times^{P_C}  \mathbb{P}^2$, where $\tilde{\mathbb{Q}}^2 \to \hat{\mathbb{Q}}^2$ is the blowup at the origin.  It follows that $\D_1 \cap \D_0 \simeq G_C \times^{P_C} \mathbb{P}^1$.  Now we can compute the stringy E-function by:
\begin{align*}
		{\rm E}_{\rm st} (\overline{\cO}_C) = {\rm E}(G_C/P_C) \times \mathbb{E} = {\rm E}({\rm IG}(2, 2l+4)) \times \mathbb{E},
	\end{align*}
	where
	\begin{align*}
		\mathbb{E} &= {\rm E}(O_2) + {\rm E}(O_1) \cfrac{q-1}{q^{2l+1}-1} + {\rm E}(\mathbb{P}^2 \setminus \mathbb{P}^1) \cfrac{q-1}{q^{3}-1} + {\rm E}(\mathbb{P}^1) \cfrac{(q-1)^2}{(q^3-1)(q^{2l+1}-1)} \\
		           &=q^2(q-1)+ \frac{(q^2-1)(q-1)}{q^{2l+1}-1}+\frac{q^2(q-1)}{q^{3}-1}+\frac{(q+1)(q-1)^2}{(q^3-1)(q^{2l+1}-1)} \\
		           &=\frac{(q-1)(q^{2l+3}-1)q^3}{(q^3-1)(q^{2l+1}-1)}.
	\end{align*}
	Thus
	\begin{align*}
		{\rm E}_{\rm st}(\overline{\cO}_C) &= \frac{(q^{2l+4}-1)(q^{2l+2}-1)}{(q-1)(q^2-1)} \times \frac{(q-1)(q^{2l+3}-1)q^3}{(q^3-1)(q^{2l+1}-1)} \\
		                         &=\frac{(q^{2l+2}-1)(q^{2l+3}-1)(q^{2l+4}-1)q^3}{(q^2-1)(q^3-1)(q^{2l+1}-1)}.
	\end{align*}
\end{proof}

As an immediate corollary, we give a negative answer to Question \ref{question}.

\begin{corollary}
 The double cover of $\overline{\0}_C$ is a mirror of $\overline{\0}_B$, but $\overline{\0}_C$ is not a mirror of the double cover of $\overline{\0}_B$.
\end{corollary}

\begin{remark}
Note that ${\rm E}_{\rm st}(\overline{\cO}_D)$ and ${\rm E}_{\rm st}(\overline{\cO}_C)$ are in general not polynomials, which gives another way to show that $\overline{\cO}_D$ and $\overline{\cO}_C$ have no crepant resolutions.
\end{remark}

\subsubsection{Asymmetry for rigid special orbits}

Let us consider the case $r=2$, namely $\overline{\0}_{2,l}^C= \cO_{[2^4, 1^{2l}]} \subset \mathfrak{sp}_{2l+8}$ and $\overline{\0}_{2,l}^B = \cO_{[3, 2^2, 1^{2l+2}]} \subset \mathfrak{so}_{2l+9}$, which are rigid special and spherical.  The double cover of $\overline{\0}_{2,l}^B$ is the orbit closure $\overline{\0}_{2,l}^D = \overline{\0}_{[2^4,1^{2l+2}]} \subset \mathfrak{so}_{2l+10}$. In the following, we will compare their stringy E-functions.

\begin{proposition}
The stringy E-function of $\overline{\0}_{2,l}^D$ is given by
$$
{\rm E}_{\rm st} (\overline{\0}^D_{2,l}) = \cfrac{(q^{l+1}+1)(q^{2l+5}+1)(q^{l+5}-1)(q^{2l+4}-1)(q^{2l+6}-1)(q^{2l+8}-1)q^{6}}{(q^2-1)(q^4-1)(q^6-1)(q^{2l+3}-1)}.
$$
\end{proposition}

\begin{proof}
In this case, $\mathfrak{n}_D \simeq \wedge^2 \mathbb{C}^4 = O_4 \bigsqcup O_2 \bigsqcup O_0$. Note that $\bar{O}_2 = O_2 \bigsqcup \{0\}$ is the affine cone of the Pl\"ucker embedding ${\rm Gr}(2,4) \subset \mathbb{P}(\wedge^2 \mathbb{C}^4)$. It follows that
$$
{\rm E}(O_2) = {\rm E}(\mathbb{C}^*) \times {\rm E}({\rm Gr}(2,4)) = (q^2+1)(q^3-1), \quad
{\rm E}(O_4) = q^6 - {\rm E}(O_2) - 1 = q^2(q-1)(q^3-1).
$$

By Proposition \ref{p.logresolutionD}, a log resolution  $\Psi: \widehat{Z} \to \overline{\0}_{2,l}^D$  is obtained by blowing up the zero section of $G_D \times^{P_D} \mathfrak{n}_D$ and we have
$$
K_{\widehat{Z}} = (2l+2) \D_2^D + 5 \D_0^D,
$$
where $\D_0^D = G_D\times^{P_D} \mathbb{P}^5$ and $\D_2^D = G_D\times^{P_D} (\0_{{\rm Gr}(2,4)}(-1))$. It follows that $\D_0^D \cap \D_2^D \simeq G_D\times^{P_D} {\rm Gr}(2,4)$. Now we can compute the stringy E-function by
$$
{\rm E}_{\rm st} (\overline{\0}^D_{2,l}) ={\rm E}({\rm OG}(4, 2l+10)) {\big (}{\rm E}(O_4) + {\rm E}(O_2) \cfrac{q-1}{q^{2l+3}-1} + {\rm E}(\mathbb{P}^5 \setminus \mathbb{Q}^4)  \cfrac{q-1}{q^{6}-1} + {\rm E}(\mathbb{Q}^4) \cfrac{(q-1)^2}{(q^6-1)(q^{2l+3}-1)}{\big )}.
$$
Now a direct computation concludes the proof by using Proposition \ref{p.EPolyGr}.
\end{proof}

\begin{remark} \label{r.leadingterm}
It's interesting to note that for $l=1, 2$,  ${\rm E}_{\rm st} (\overline{\0}^D_{2,l})$ is a polynomial, which is not the case for $l \geq 3$. When $l=1$, we have
$$
{\rm E}_{\rm st} (\overline{\0}^D_{2,1}) = \cfrac{(q^5+1)(q^6-1)(q^7+1)(q^8-1)q^6}{(q^2-1)^2},
$$
which has the symmetric expansion
$$
q^{28}+2q^{26}+3q^{24}+q^{23}+3q^{22}+3q^{21}+2q^{20}+5q^{19}+q^{18}+6q^{17}+q^{16}+5q^{15}+2q^{14}+3q^{13}+3q^{12}+q^{11}+3q^{10}+ 2q^8+q^6.
$$
\end{remark}

We now turn to the case $\overline{\0}_{2, l}^C$ and start with some preliminary results.

\begin{proposition} \label{p.EPolyOrbits}
For $0 \leq i \leq m$, let $O_i(\C^m)\subset {\rm Sym}^2 (\mathbb{C}^m)$ be the set of symmetric matrices of rank $i$. Then
 $${\rm E}(O_k(\C^m)) = {\rm E}({\rm Gr}(k, m)) \times {\rm E}(O_k(\C^k)), \text{and} \ q^{\frac{m(m+1)}{2}} = \sum_{k=0}^m {\rm E}(O_k(\C^m)).$$
\end{proposition}
\begin{proof}
We consider the fibration $O_k(\C^m) \to {\rm Gr}(m-k, m)$ by sending $A \in O_k(\C^m)$ to ${\rm Ker}(A) \in {\rm Gr}(m-k, m)$, which is Zaisiki locally trivial (see for example the proof of \cite[Proposition 4.4]{BL19}). The fiber of this fibration is $O_k(\C^k)$, which shows that ${\rm E}(O_k(\C^m)) = {\rm E}({\rm Gr}(k, m)) \times {\rm E}(O_k(\C^k))$.

On the other hand, we have
$$
q^{\frac{m(m+1)}{2}} = {\rm E}({\rm Sym}^2\mathbb{C}^m) = \sum_{k=0}^m {\rm E}(O_k(\C^m)).
$$
\end{proof}

By a simple computation, we have the following explicit formulae:
\begin{corollary}\label{c.EPolyO}
Let $O_i\subset {\rm Sym}^2 \mathbb{C}^4$ be the set of symmetric matrices of rank $i$, $i=1,2,3,4$.
Then their E-polynomials are given by
\begin{align*}
{\rm E}(O_1) &= q^4-1,  \quad \quad   \quad \quad  \quad \quad  {\rm E}(O_2) = q^2(q^2+1)(q^3-1), \\  {\rm E}(O_3) &= q^2(q^3-1)(q^4-1),  \quad \quad  {\rm E}(O_4) =q^6(q-1)(q^3-1).
\end{align*}
\end{corollary}
\begin{remark}
It seems that the general formulae are given by the following (which we do not need):
$$
{\rm E}(O_{2r}(\C^{2r})) = \prod_{j=1}^r (q^{2j-1}-1)q^{2j}, \quad  {\rm E}(O_{2r+1}(\C^{2r+1}))  = \prod_{j=0}^r (q^{2j+1}-1)q^{2j}.
$$
\end{remark}

%We denote by $\mathbb{P}(O_i)$ the image of $O_i(\C^m)$ in $\mathbb{P}({\rm Sym}^2\C^m)$, then $O_i(\C^m) \to \mathbb{P}(O_i)$ is a $\C^*$-bundle. Let $W_m \to \mathbb{P}({\rm Sym}^2\C^m)$ be the birational map obtained by successive blowups along strict transforms of closures of $\mathbb{P}(O_i)$ from the smallest to the biggest.  Note that $W_m$ is the classical variety of complete quadrics (cf. \cite{Th99}).
%
%\begin{lemma}
%Let $\cP_m$ be the E-polynomial of $W_m$. Then we have the following recursive  relation:
%$$
%\cP_m = {\rm E}(\mathbb{P}(O_m)) + {\rm E}(\mathbb{P}(O_{m-1})) \cP_1 + \cdots +  {\rm E}(\mathbb{P}(O_1)) \cP_{m-1}.
%$$
%\end{lemma}
%\begin{proof}
%This follows from the fact that the map $W_m \to \mathbb{P}({\rm Sym}^2\C^m)$ is a locally trivial fibration over each stratum $\mathbb{P}(O_i)$ with fiber being $W_{m-i}$.
%\end{proof}
%
%It would be interesting to work out the general formula for $\cP_n$.  For our purpose, we only need the first few terms, which can be deduced by direct computation.
%\begin{corollary} \label{c.CP4}
%\begin{align*}
%\cP_1 &=1,\quad  \cP_2=\cfrac{q^3-1}{q-1}, \quad \cP_3=\cfrac{(q^3-1)(q^4-1)}{(q-1)^2} \\  \cP_4 &=\cfrac{(q^2-1)(q^3-1)(q^7+q^5+q^4-q^3-q^2-1)}{(q-1)^3}.
%\end{align*}
%\end{corollary}

\begin{proposition} \label{p.asymmetryRS}
The orbit closure $\overline{\0}^C_{2,1} =\overline{\0}^C_{[2^4, 1^2]} \subset \mathfrak{sp}(10)$ is not a mirror of the double cover of $\overline{\0}^B_{2,1} = \overline{\0}^B_{[3,2^2, 1^4]} \subset \mathfrak{so}(11)$.
\end{proposition}
\begin{proof}
We have $\mathfrak{n}_D ={\rm Sym}^2 \C^4$. Let $O_i \subset {\rm Sym}^2 \mathbb{C}^4$ be the set of elements of rank $i$. By Proposition \ref{p.logresolutionD}, we have a log resolution  $\Phi: \widehat{Z} \to \overline{\0}_{2,l}^D$ with discrepancies given by
$$
K_{\widehat{Z}} = 2 \D_3 + 2 \D_2 + 5 \D_1 + 9 \D_0.
$$

%Note that the map $\Phi$ restricted to each $O_i$ is a fibration with fiber $W_{4-i}$. This gives that
By definition, we have
$$
{\rm E}_{\rm st} (\overline{\0}^C_{2,1}) ={\rm E}({\rm IG}(4,10)) \times \mathbb{E},
$$
where
$$\mathbb{E} = {\rm E}(O_4) + {\rm E}(\D_0^\circ) \cfrac{q-1}{q^{10}-1} + {\rm E}(\D_1^\circ) \cfrac{q-1}{q^6-1} + {\rm E}(\D_2^\circ) \cfrac{q-1}{q^3-1} + {\rm E}(\D_3^\circ) \cfrac{q-1}{q^{3}-1} + \overline{E},
$$
and the last term $\bar{E}$ comes from contributions of intersections of exceptional divisors.

We will only consider the leading terms of ${\rm E}_{\rm st} (\overline{\0}^C_{2,1}) $. Recall that for an $m$--dimensional smooth variety, its E--polynomial has $q^m$ as the leading term. This implies that the leading terms in $\overline{E}$ are at most $q^4$.  Similarly, the leading terms coming from divisors
${\rm E}(\D_i^\circ) \cfrac{q-1}{q^{a_i+1}-1}$ are at most $q^7$.  This shows that the leading terms of ${\rm E}_{\rm st} (\overline{\0}^C_{2,1}) $ is given by
$$
{\rm E}({\rm IG}(4,10)) \times (q^{10}-q^9+O(q^7)).
$$
By Proposition \ref{p.EPolyGr}, we have
$$
{\rm E}({\rm IG}(4,10)) = \cfrac{(q^6-1)(q^8-1)(q^{10}-1)}{(q-1)(q^2-1)(q^3-1)} = q^{18}+q^{17}+2q^{16}+3q^{15}+4q^{14}+\cdots.
$$

It follows that the leading terms of ${\rm E}_{\rm st} (\overline{\0}^C_{2,1}) $ is given by
$$
q^{28}+q^{26}+\cdots,
$$
which is different from the leading terms of ${\rm E}_{\rm st} (\overline{\0}^D_{2,1}) $ (see Remark \ref{r.leadingterm}).
\end{proof}

\begin{remark}
By a more involved computation, we can actually show that $\overline{\0}^C_{2,l}$ is not a mirror of the double cover of $\overline{\0}^B_{2,l}$ for all $l \geq 1$.
\end{remark}
%\begin{remark}
%The stringy E-function of $\overline{\cO}^C_{2,l}$ can also be worked out as follows, but the formula is not very nice. We then can check the asymmetry holds for all $l$.
%\begin{align*}
%		{\rm E}_{{\rm st}}(\overline{\cO}^C_{2,l}) = \frac{(F_1+F_2)(q^{2l+2}-1)(q^{2l+4}-1)(q^{2l+6}-1)(q^{2l+8}-1)}{(q-1)(q^2-1)(q^4-1)(q^6-1)^2(q^{10}-1)(q^{2l+1}-1)},
%	\end{align*}
%	where
%	\begin{align*}
%		& F_1 = (q+1)(q^3-1)(q^6-1)(q^{2l+1}-1)[q^3+(q^2+1)(q^4+q^3+q^2+q+1)], \\
%		& F_2 = (q-1)(q^2-1)\left(q^{2l+22}(q^2+1)+\right.  \\
%		& (q^{2l+1}-1)(2q^{19}+2q^{18}+3q^{17}+5q^{16}+6q^{15}9q^{14}+5q^{13}+6q^{12}\\
%		& +q^{11}+3q^{10}-2q^9+q^8-2^7-3q^6-5q^5-9q^4-6q^3-7q^2+2q+1) \\
%		&  -q^{13} -q^{12} -q^{11} +q^8 +q^7 +q^5 -q^4 -q^2  \big).
%	\end{align*}
%\end{remark}

\end{document}